\documentclass[reqno,11pt]{amsart}

\usepackage{amsmath,amsfonts,amssymb,amsthm,epsfig}
\usepackage{mathtools,enumerate}
\usepackage[english]{babel}
\usepackage{tikz}
\usetikzlibrary{arrows}
\usepackage{comment}

 \allowdisplaybreaks
\numberwithin{equation}{section}

\font\script=rsfs10 at 11pt
\def\eps{\varepsilon}
\def\F{{\mathcal F}}
\def\G{{\mathcal G}}
\def\H{{\mbox{\script H}\,\,}}
\def\Bar{{\rm bar}}
\def\A{{\mathcal A}}
\def\R{\mathbb R}
\def\RR{\mathcal R}
\def\S{\mathbb S}
\def\N{\mathbb N}
\def\INT{{\rm Int}}
\def\bal{\begin{aligned}}
	\def\eal{\end{aligned}}
\def\proofof#1{\begin{proof}[Proof of #1]}
	\def\Chi#1{\hbox{{\large $\chi$}{\Large $_{_{#1}}$}}}
	
	\def\step#1#2{\par\noindent{\underline{\it Step~#1.}}\emph{ #2}\\}
	
	\def\XXint#1#2#3{{\setbox0=\hbox{$#1{#2#3}{\int}$} \vcenter{\vspace{-1pt}\hbox{$#2#3$}}\kern-.5\wd0}}

	\def\comp{\subset\subset}

	\pgfarrowsdeclare{<<<}{>>>}
	{
		\arrowsize=0.2pt
		\advance\arrowsize by .5\pgflinewidth
		\pgfarrowsleftextend{-4\arrowsize-.5\pgflinewidth}
		\pgfarrowsrightextend{.5\pgflinewidth}
	}
	{
		\arrowsize=0.2pt
		\advance\arrowsize by .5\pgflinewidth
		\pgfpathmoveto{\pgfpointorigin}
		\pgfpathlineto{\pgfpoint{-1.2mm}{-.8mm}}
		\pgfusepathqstroke
		\pgfpathmoveto{\pgfpointorigin}
		\pgfpathlineto{\pgfpoint{-1.2mm}{.8mm}}
		\pgfusepathqstroke
	}%%%% DEFINIZIONE FRECCIA CARINA:FINE (per usarla, iniziare una figura con \begin{tikzpicture}[>=>>>] )
		
		\newcounter{mt}
		\def\maintheorem#1#2#3{\par \medskip \noindent {\bf Theorem~\mref{#1}}~(#2).~{\it #3}\par}
		\def\mref#1{\Alph{#1}}
		\def\maintheoremdeclaration#1{\stepcounter{mt}\newcounter{#1}\setcounter{#1}{\arabic{mt}}}
		\maintheoremdeclaration{exist}
		\newtheorem{theorem}{Theorem}[section]
		\newtheorem{lemma}[theorem]{Lemma}
		\newtheorem{prop}[theorem]{Proposition}
		
		\newtheorem{defin}[theorem]{Definition}
		\newtheorem{remark}[theorem]{Remark}
		
\begin{document}
			
			\title[Minimizer of the barycentric isoperimetric ratio in $\R^2$]{On the minimizer of the barycentric isoperimetric ratio among connected planar sets}
			
			\author{C. Gambicchia}
			\address{Scuola Normale Superiore\newline \indent
				Piazza dei Cavalieri, 7, 56127 Pisa, Italy.}
			\email{\tt chiara.gambicchia@sns.it}

			\author{A. Pratelli}
			\address{Dipartimento di Matematica \newline \indent Università di Pisa\newline \indent
				Largo Bruno Pontecorvo, 5, 56127 Pisa, Italy.}
			\email{\tt aldo.pratelli@unipi.it}
			
			\author{C. Torres-Latorre}
			\address{Instituto de Ciencias Matem\'aticas \newline \indent Consejo Superior de Investigaciones Cient\'ificas\newline \indent
				C/ Nicol\'as Cabrera, 13-15, 28049 Madrid, Spain.}
			\email{\tt clara.torres@icmat.es}
			
			\begin{abstract}
				Among connected, planar sets, the ratio between the isoperimetric deficit and the squared barycentric asymmetry is known to be bounded from below by a strictly positive constant (as proved in the recent papers~\cite{BCH,GP}). This opens the natural question whether the infimum is attained at some set. In this paper, we prove that a minimizer for this ratio actually exists, and moreover it is simply connected.
			\end{abstract}

\newcommand{\CRIT}{\textbf{\textcolor{red}{[CRITICAL]}} }
\newcommand{\ERR}{\textbf{\textcolor{red!70!black}{[ERROR]}} }
\newcommand{\GAP}{\textbf{\textcolor{orange!80!black}{[GAP]}} }
\newcommand{\GLOSS}{\textbf{\textcolor{blue!60!black}{[GLOSS]}} }
\newcommand{\TYPO}{\textbf{[TYPO]} }
\newcommand{\OK}{\textbf{\textcolor{green!50!black}{[VERIFIED]}} }
\newcommand{\FIX}[1]{\emph{Fix: #1}}

			\maketitle

			\section{Introduction}
			
			Quantitative forms of isoperimetric inequalities have been intensively studied for a long time, and particularly in the last 20 years. The meaning of such inequalities is very simple. Namely, it is known that a set which minimizes the perimeter among those with fixed area is a ball; a quantitative version of this information says that a set which \emph{almost} minimizes the perimeter must be \emph{very close} to a ball. To state it precisely, we define the \emph{isoperimetric deficit} of a set $E\subseteq \R^N$ as
			\[
			\delta(E) = \frac{P(E) - P(B^E)}{P(B^E)}\,,
			\] 
			where $B^E$ is any ball with the same volume as $E$. Then, we need to consider a distance to measure how close $E$ is from being a ball; the most used such distance is the so-called \emph{Fraenkel asymmetry}, defined as
			\begin{equation}\label{deflambda}
				\lambda(E) = \min_{x\in\R^N} \frac{|E\Delta (x+B^E)|}{|E|}\,.
			\end{equation}
			The meaning of this definition is clear; that is, the Fraenkel asymmetry of a set $E$ is the volume of the symmetric difference between $E$ and the ``best translate'' of the ball $B^E$ (the fact that the infimum is actually a minimum is very simple to prove). Using this notion, the sharp quantitative isoperimetric inequality (see~\cite{FMP,CL,FiMP}) says that there exists some constant $C_F$ (only depending on the dimension $N$) such that for every set $E$ of finite perimeter one has
			\begin{equation}\label{qii}
				\lambda(E)\leq C_F \sqrt{\delta(E)}\,.
			\end{equation}
			However, the Fraenkel asymmetry is not the only possible choice for measuring how far $E$ is from being a ball. Its main problem is that it is not easy to compute, since it depends on which is the ``optimal ball'' realizing the minimum in~(\ref{deflambda}). Another possible choice is to use, instead, the \emph{barycentric asymmetry}, given by
			\[
			\lambda_0(E) = \frac{|E \Delta B_E|}{|E|}\,,
			\]
			where $B_E$ is the ball having the same volume as $E$ centered at the barycenter of $E$. On the one hand, this is a somewhat arbitrary choice, which is possible only for sets for which the barycenter is defined (for instance, bounded sets); but on the other hand, making this choice solves the uncertainty about the optimal ball, and in all the simple examples such a choice seems quite reasonable. In this regard, it is interesting to note that the barycentric asymmetry was introduced, even before the Fraenkel one, in the 1990's by Fuglede (see~\cite{F93,F89}). And moreover, even in cases where the Fraenkel asymmetry is used, the estimates are actually obtained by using the much easier to compute barycentric one, thanks to the obvious estimate $\lambda(E)\leq \lambda_0(E)$.
			
			The inequality analogous to~(\ref{qii}) for the barycentric asymmetry, that is,
			\[
			\lambda_0(E)\leq C \sqrt{\delta(E)}\,,
			\]
			is easily seen to be false for general sets of finite perimeter $E\subseteq\R^N$. However, it was proved by Fuglede in~\cite{F93} that the inequality is actually true, with a constant $C$ only depending on $N$, if one restricts himself to consider \emph{convex sets}, for which in particular the barycenter is always well defined. Many years later, in~\cite{BCH} it was proved that the same inequality is valid also for \emph{connected sets in $\R^2$}. Then, very recently, in the paper~\cite{GP} it was proved that the same inequality is again true for all \emph{bounded sets in $\R^N$}, with a constant depending on the dimension $N$ and on the diameter of the set $E$; as remarked in~\cite{GP}, this last result generalizes both previously mentioned cases. It is convenient to rewrite the inequality introducing the ratio $\G$ between the isoperimetric deficit and the square of the barycentric asymmetry; that is, the inequality reads as
			\begin{equation}\label{qbi}
				\G(E) := \frac{\delta(E)}{\lambda_0(E)^2} \geq c
			\end{equation}
			for some strictly positive constant $c>0$. A natural question then arises, namely, whether there is some set $E$ minimizing the ratio $\G$ in some class. In this paper, we consider this question for the simplest case, namely, the connected sets in $\R^2$. It is interesting to notice that the analogous question was posed also for the case of the inequality~(\ref{qii}); in fact, a byproduct of the proof of~(\ref{qii}) given in~\cite{CL} is that there exists a set minimizing the ratio $\delta(E)/\lambda(E)^2$ in the plane; in the paper~\cite{CL2} several properties of this optimal set have been found, and a candidate has been exhibited, but proving that this set is actually a minimizer is still an open question after more than $10$ years.
			
			The situation for the barycentric case has many similarities. The goal of this paper is to show that a minimizer for $\G$ in the class of planar, connected sets actually exists, and to study some of its properties; in particular, we will show that the minimizer must be without holes. Also for this case there is a candidate minimizer, conjectured in~\cite{BCH}, and this is the ``modified dumbbell'' $D$ shown in Figure~\ref{fig:dumbbell}. As one can see, the fact that this set is ``connected'' depends on what definition one wants to use, as well as the exact value of the perimeter. The precise definitions will be given in Section~\ref{secdefin}; however, roughly speaking we can say that the set $D$ will be considered as connected since it is the interior of a connected, non self-crossing curve, and its perimeter is the length of this curve (so, in particular the horizontal segment has to be counted twice). The main result of this paper is then the following one.
			
			\maintheorem{exist}{Existence of a minimizer for $\G(E)$}{There exists a minimizer of the ratio $\G(E)$ in the class of connected, planar sets. Moreover, any minimizer is simply connected.}
			
			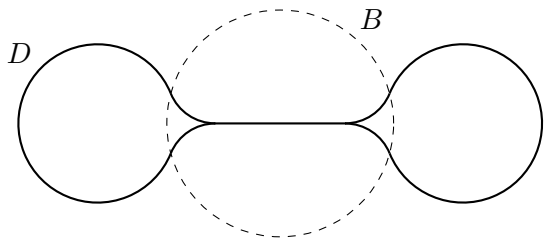
\begin{figure}[thbp]%questo il dumbbell, piuttosto utile, me lo salvo.
				\begin{tikzpicture}[scale=1.5]
\draw[dashed] (0,0) circle (1);
\draw[thick] (-0.965,0.264) arc[start angle=22.24, end angle=337.76, radius=0.698];
\draw[thick] (-0.965,0.264) arc[start angle=202.24, end angle=270, radius=0.424];
\draw[thick] (-0.965,-0.264) arc[start angle=157.76, end angle=90, radius=0.424];
\draw[thick] (0.965,-0.264) arc[start angle=-157.76, end angle=157.76, radius=0.698];
\draw[thick] (0.965,0.264) arc[start angle=-22.24, end angle=-90, radius=0.424];
\draw[thick] (0.965,-0.264) arc[start angle=22.24, end angle=90, radius=0.424];
\draw[thick] (-0.58,0) -- (0.58,0);
\draw (-2.1,.8) node[anchor=north east] {$D$};
\draw (1,1.1) node[anchor=north east] {$B$};
		\end{tikzpicture}
				\caption{The ``modified dumbbell'', conjectured as minimizer of $\G$.}\label{fig:dumbbell}
			\end{figure}
			
The plan of the paper is the following. After the short Section~\ref{secdefin}, in which we fix the main notation and in particular we discuss the precise meaning of what a ``connected set'' and its perimeter are, in Section~\ref{secestidumb} we give an explicit estimate of the functional for the candidate minimizer. 
After that, we prove a very important a-priori formula for the curvature of the boundary of a minimizer in Section~\ref{seccurv}. Then, we divide the proof of Theorem~\mref{exist} in two parts. In Section~\ref{secexist} we prove the existence of a minimizer for $\G$; then, we prove the simple connectedness in Section~\ref{minsimcon}. Finally, in Appendix~\ref{technical appendix}, the reader will find some technical results used in the proof of the simple connectedness.

\subsection{Notation and basic definitions\label{secdefin}}

Since in this paper we want to deal with connected sets in $\R^2$ and their perimeter, it is important to describe precisely the meaning of these words. Indeed, the usual De Giorgi notion of perimeter is measure theoretic, hence it does not distinguish between sets which are equal almost everywhere; on the other hand, the topological notion of connectedness does not have the same property, because it is clearly possible that two sets are equal almost everywhere, one of them being topologically connected and the other one not. There are then two main possibilities. The first one is to use the usual De Giorgi notion of perimeter and, correspondingly, the measure-theoretic notion of connectedness; namely, a set $E$ is connected (in the measure-theoretical sense) if and only if, whenever we write $E=E_1 \cup E_2$ with $P(E)=P(E_1)+P(E_2)$, then one of $E_1$ and $E_2$ must be negligible. The second one is to use the usual topological notion of connectedness, and the Minkowski notion of perimeter. Notice that making this choice is not necessary at all if one is interested in proving the inequality~(\ref{qbi}) for ``connected'' sets, because with both choices it is immediate to reduce oneself to consider smooth sets, and for smooth sets the two possible notions of perimeter and the two possible notions of connectedness coincide. Instead, if one wants to study the existence of minimizers for the functional $\G$, then the choice is clearly relevant. In this paper, we opt for the second possibility, that is, the standard topological connectedness and the Minkowski perimeter (which will be precisely described below). Let us consider for a moment the ``modified dumbbell'' $D$ of Figure~\ref{fig:dumbbell}; as said above, the conjecture made in~\cite{BCH} is that this set is the only minimizer of the functional, and the results of the present paper are compatible with the validity of the conjecture. Notice that $D$ is topologically connected, but it is not connected in the measure-theoretical sense, because it is made of two different pieces, joined with a segment, which has zero area. In addition, the Minkowski perimeter of $D$ is the total length of the closed, oriented curve $\partial D$, so the length of the segment is counted twice; instead, the De Giorgi perimeter of $D$ is lower, because the segment is not counted at all (since from the measure-theoretical point of view it does not exist). Observe that, if the conjecture is true, then a byproduct is that there is no minimizer of $\G$ in the measure-theoretical sense.

Let us quickly describe the Minkowski notion of perimeter in the plane. Consider a closed curve, that is, a continuous map $\tau: \S^1\to\R^2$; we say that $\tau$ is \emph{non-crossing} if it is the uniform limit of injective curves $\tau_j :\S^1\to\R^2$. For every such curve there is an \emph{internal part}, that we call $\INT(\tau)$, which consists of all the points of degree $1$ or $-1$ with respect to the curve; it is standard to observe that all the points of $\R^2\setminus \tau(\S^1)$ have degree either $0$ or $\pm 1$, and in fact either the degree is always $0$ or $1$, or it is always $0$ or $-1$ (the first case occurs if the curve is counterclockwise, the second one if it is clockwise). Then, it is standard to prove that for any closed, connected planar set $E\subseteq\R^2$ of finite perimeter there exist some non-crossing curves $\tau_i:\S^1\to\R^2$ (either finitely many or countably many) in such a way that $E$ coincides with
\[
\left(\tau_1(\S^1) \cup \INT(\tau_1)\right) \setminus \bigcup_{i>1} \INT(\tau_i)\,,
\]
that for every $i\neq j$ the two open sets $\INT(\tau_i)$ and $\INT(\tau_j)$ are either disjoint or contained one into the other, and that $\tau_1(\S^1) \cup \INT(\tau_1)$ is simply connected. Moreover, the Minkowski perimeter of $E$ coincides with
\[
P(E) = \sum_i \int_{\S^1} |\tau_i'|\, d\H^1\,.
\]
In words, the sets are obtained by taking the internal part of the ``external curve'' $\tau_1$, and cutting away the internal parts of the curves $\tau_i$ with $i>1$ (so, these latter internal parts are ``holes'' in the set). In particular, a connected set $E$ is simply connected if there is just a single curve. Keep in mind that our main result says that a minimizer of the barycentric isoperimetric ratio among the connected sets exists, and it is simply connected. Notice that in general $\int_{\S^1} |\tau_i'|$ is not the length of the curve $\tau_i$, since a non-crossing curve can pass twice through a same subpath, and in that case the length of the path is counted twice. For instance, the set $D$ of Figure~\ref{fig:dumbbell} is simply connected, and the length of the horizontal segment is counted twice in the perimeter.\par

Thanks to the above discussion, dealing with connected, planar sets reduces to dealing with the curves $\tau_i$, since the connectedness, the simple connectedness, and the perimeter of any set can be studied directly through the curves. We conclude this short section with the following definition of ``simple point''.

\begin{defin}[Simple point]
Let $E$ be a closed connected set of finite perimeter in $\R^2$, and let $\tau_i$ be curves identifying $E$ as before. A point $p\in\partial E$ is said to be a \emph{simple point} if it is contained in a single curve $\tau_i$ and there is exactly one $t\in \S^1$ such that $\tau_i(t) = p$.
\end{defin}

\section{An explicit estimate}\label{secestidumb}

In this short section, we give an \emph{a priori} upper bound to the optimal value of the barycentric ratio. This bound will then allow us to exclude some competitors in the rest of the paper.
\begin{prop}\label{prop:estim_mod_dumb}
For a minimizer $E$ of the functional $\G$, it holds $\G(E)<1/4$.
\end{prop}
\begin{proof}
The proof consists basically of an explicit estimate of the functional for a modified dumbbell.
Consider a parameter $\beta$, which will be the curvature of the dumbbell outside the unit ball.
Given an angle $\vartheta\in [0,\pi/4]$, we can denote by $D_{\beta,\vartheta}$ the dumbbell whose top-right quarter is represented in Figure~\ref{fig:quarter of dumbbell} (namely, the set $D_{\beta, \vartheta}$ is the two-symmetric set obtained starting from the set represented in the picture).

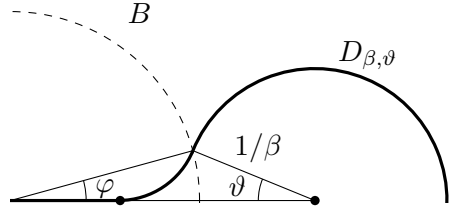
\begin{figure}[!ht]
\centering
\begin{tikzpicture}[scale=2.5]
\begin{scope}
\clip (0,-0.01) rectangle (2.5, 1.2);
\draw[dashed] (0,0) circle (1);
\draw[very thick] (0.965,-0.264) arc[start angle=-157.76, end angle=157.76, radius=0.698];
\draw[very thick] (-0.58,0) -- (0.58,0);
\end{scope}
\fill (0.58,0) circle (.7pt);
\draw[very thick] (0.965,0.264) arc[start angle=-22.24, end angle=-90, radius=0.424];
\draw (2.1,0.9) node[anchor=north east] {$D_{\beta,\vartheta}$};
\draw (0.8,1.1) node[anchor=north east] {$B$};
\draw (0,0)--(0.965,0.264);
\draw (1.611,0)--(0.965,0.264);
\begin{scope}
\clip (0,0)--(0.965,0.264)--(1.611,0)--(0,0);
\draw (0,0) circle (0.4);
\draw (1.611,0) circle (0.3);
\end{scope}
\draw (0,0) -- (1.611,0);
\fill (1.611,0) circle (.7pt);
\draw (1.31,0.28) node {$1/\beta$};
\draw (1.2,0.08) node {$\vartheta$};
\draw (0.5,0.065) node {$\varphi$};
\end{tikzpicture}
\caption{Representation of the first-quadrant quarter of a dumbbell.}
\label{fig:quarter of dumbbell}
\end{figure}

In particular, we can express the internal curvature in terms of $\beta$ and $\vartheta$, simply imposing that the arc of circle which is inside the unit ball is tangent to the horizontal axis, thus finding
\[
r_{in}=\frac{\sin \vartheta}{\beta(1-\sin\vartheta)}.
\]
Moreover, we can also express the area of $D_{\beta, \vartheta}$ in terms of $\beta$ and $\vartheta$ as
\begin{equation}\label{eq: area dumbbell}
|D_{\beta,\vartheta}|=2\left(\frac{1}{\beta^2}(\pi-\vartheta)+\left(r_{in}+\frac{1}{\beta}\right)^2\sin\vartheta\cos\vartheta-r_{in}^2\left(\frac{\pi}{2}-\vartheta\right)\right).
\end{equation}
Now we can show that, for any given $\beta\in\left(\sqrt{2},3/2\right)$, there exists a value of $\vartheta$ such that $|D_{\beta, \vartheta}|=\pi$.
Indeed, plugging the expression for $r_{in}$ into~\eqref{eq: area dumbbell} we obtain
\[
\begin{aligned}
|D_{\beta,\vartheta}|=& 2\left(\frac{1}{\beta^2}(\pi-\vartheta)+\left(\frac{\sin \vartheta}{\beta(1-\sin\vartheta)}+\frac{1}{\beta}\right)^2\sin\vartheta\cos\vartheta-\left(\frac{\sin \vartheta}{\beta(1-\sin\vartheta)}\right)^2\left(\frac{\pi}{2}-\vartheta\right)\right)\\
=&\frac{2}{\beta^2}\left(\pi -\vartheta+\frac{\sin\vartheta\cos\vartheta}{(1-\sin\vartheta)^2}-\frac{\sin^2\vartheta}{(1-\sin\vartheta)^2}\left(\frac{\pi}{2}-\vartheta\right)\right).
\end{aligned}
\]
Now, computing this for $\vartheta=0$ and $\vartheta=\pi/5$, we find
\[
|D_{\beta,0}|=\frac{2\pi}{\beta^2},\qquad |D_{\beta,\frac{\pi}{5}}|=\frac{2}{\beta^2}\left(\frac{4\pi}{5}+\frac{\sin(\pi/5)}{\big(1-\sin(\pi/5)\big)^2}\left(\cos(\pi/5)-\frac{3\pi}{10}\sin(\pi/5)\right)\right),
\]
which, for $\beta\in(\sqrt{2}, 1.47)$, give us $|D_{\beta,0}|<\pi$ and $|D_{\beta,\pi/5}|>\pi$.
By continuity of $|D_{\beta, \vartheta}|$, we can conclude that for at least one value of $\vartheta$, that we call $\vartheta_0$, we have the desired area.
Let us denote by $D_{\beta}=D_{\beta, \vartheta_0}$ such set, that satisfies the area constraint.

Then, we can express and compute the functional $\G$ in terms of $\beta$ and $\vartheta_0$.
Let us start with the perimeter of $D_\beta$.
First, we introduce an auxiliary variable, that is, the angle $\varphi$ in Figure~\ref{fig:quarter of dumbbell}, which is given by
\[
\varphi=\arcsin\left(\frac{1}{\beta}\sin\vartheta_0\right).
\]
Now some easy trigonometrical computations lead us to the following formula for the perimeter:
\begin{equation*}
P(D_{\beta})=4\cdot\left(\frac{1}{\beta}(\pi -\vartheta_0)+r_{in}\left(\frac{\pi}{2}-\vartheta_0\right)+\cos\varphi-r_{in}\cos\vartheta_0 \right),
\end{equation*}
where, inside the parentheses, the first term is the length of the outer arc of circle, the second one is the length of the inner arc of circle and the rest is the length of the segment connecting the arc to the origin.
Hence, the isoperimetric deficit of $D_{\beta}$ is given by
\begin{equation}\label{eq: deficit dumbbell}
\delta(D_{\beta})=\frac{2}{\pi}\left(\frac{1}{\beta}(\pi -\vartheta_0)+r_{in}\left(\frac{\pi}{2}-\vartheta_0\right)+\cos\varphi-r_{in}\cos\vartheta_0 \right)-1.
\end{equation}
Now we can compute the barycentric asymmetry.
It is easy to compute the area of $D_{\beta}\setminus B$ as 
\[
|D_{\beta}\setminus B|=4\cdot \left(\frac{1}{2\beta^2}(\pi-\vartheta_0)+\frac{1}{2\beta}\sin\vartheta_0\left(\cos\varphi+\frac{1}{\beta}\cos\vartheta_0\right)-\frac{\varphi}{2}\right).
\]
Hence, taking into account that $|D_{\beta}\Delta B|=2|D_{\beta}\setminus B|$, we have
\begin{equation}\label{eq: asimmetria dumbbell}
\lambda_0(D_{\beta})=\frac{8}{\pi}\left(\frac{1}{2\beta^2}(\pi-\vartheta_0)+\frac{1}{2\beta}\sin\vartheta_0\left(\cos\varphi+\frac{1}{\beta}\cos\vartheta_0\right)-\frac{\varphi}{2}\right).
\end{equation}
Finally, we choose $\beta=1.43$; then, an admissible value of $\vartheta_0$ (that is, a value such that $|D_{\beta,\vartheta}|=\pi$) is $\vartheta_0\approx 21^\circ$. With this choice of $\beta$ and $\vartheta_0$ (which are the values used to plot Figure~\ref{fig:dumbbell}), \eqref{eq: deficit dumbbell} and~\eqref{eq: asimmetria dumbbell} give $\delta(D_{\beta})\approx 0.919$ and $\lambda_0(D_\beta)\approx 1.922$, thus finding $\G(D_\beta)\approx 0.2488<1/4$.
\end{proof}
% \begin{figure}[!ht]
% \centering
% \begin{tikzpicture}[scale=1.2]
% \draw[dashed] (0,0) circle (1);
% \draw[thick] (-0.965,0.264) arc[start angle=22.24, end angle=337.76, radius=0.698];
% \draw[thick] (-0.965,0.264) arc[start angle=202.24, end angle=270, radius=0.424];
% \draw[thick] (-0.965,-0.264) arc[start angle=157.76, end angle=90, radius=0.424];
% \draw[thick] (0.965,-0.264) arc[start angle=-157.76, end angle=157.76, radius=0.698];
% \draw[thick] (0.965,0.264) arc[start angle=-22.24, end angle=-90, radius=0.424];
% \draw[thick] (0.965,-0.264) arc[start angle=22.24, end angle=90, radius=0.424];
% \draw[thick] (-0.58,0) -- (0.58,0);
% \draw (-2.1,.8) node[anchor=north east] {$D$};
% \draw (1,1.1) node[anchor=north east] {$B$};
% 		\end{tikzpicture}
% \caption{A quite-accurate drawing of the dumbbell $D_{1.43}$.}
% \label{fig:dumbbell 1}
% \end{figure}

\section{The formula for the curvature}\label{seccurv}

This section is devoted to show a formula for the curvature of the boundary of a minimizer of the ratio $\G$. Notice that this is an \emph{a priori} result, since up to now we do not know that a minimizer actually exists. The result is the following.
\begin{theorem}\label{formulaabc}
Let $E$ be a minimizer of $\G$ among connected sets of area $\pi$ and barycenter in the origin. Then, each path $\tau_i:\S^1\to\R^2$ is a ${\rm C}^{1,1}$ curve and there are three constants $\alpha,\, \beta,\,\gamma$ such that, up to a rotation, the curvature $\kappa(x,y)$ at any simple point $(x,y)\in\partial E\setminus \partial B$ satisfies
\begin{equation}\label{abc}
\kappa(x,y)=\alpha x + \beta - \gamma \Chi{B}\,,
\end{equation}
where the constant $\gamma$ is given by
\begin{equation}\label{formulagamma}
\gamma= 8\, \frac{\delta (E)}{\lambda_0(E)}=8\,\lambda_0(E)\, \G (E)\,.
\end{equation}
\end{theorem}
\begin{remark}
It is not hard to show that a minimizer $E$ of $\G$ is a dumbbell (possibly without segment) if and only if the constant $\alpha$ in~\eqref{abc} vanishes.
Indeed, it is clear that a dumbbell satisfies the condition $\alpha=0$ because it must hold $\kappa(x,y)=\kappa(-x,y)$.
Reversely, by a standard use of the so-called double spherical cap symmetrization (see e.g.~\cite{G}), it can be easily proved that the condition $\alpha=0$ implies that the set $E$ is a dumbbell.
\end{remark}

To prove this result, we will rely on a technical result, namely, we show that the set $E$ must be a $\Lambda$-minimizer. We briefly recall here the (standard) definition of the $\Lambda$-minimality.
\begin{defin}[$\Lambda$-minimality]
A set $E\subseteq\R^N$ of finite perimeter is said to be a \emph{$\Lambda$-minimizer at a point $x\in\R^N$} for a constant $\Lambda>0$ if there exists $r_0=r_0(x)>0$ such that for every ball $B(x,r)$ with $r<r_0$ and for every set $F$ such that $F\Delta E \comp B(x,r)$, one has $P(E)\leq P(F)+ \Lambda r^N$. The set $E$ is said to be a \emph{$\Lambda$-minimizer} if it is a $\Lambda$-minimizer at \emph{every} point $x\in\partial E$.
\end{defin}
The exponent $N$ in $r^N$ specializes to $2$ throughout this paper, since we work in $\R^2$. We can then prove the $\Lambda$-minimality of $E$ at points of the reduced boundary.

\begin{lemma}\label{Lambda-min lemma}
Let $E\subseteq\R^2$ be as in Theorem~\ref{formulaabc}. Then $E$ is a $\Lambda$-minimizer at all points of $\partial^* E$.
\end{lemma}
\begin{proof}
Let $E$ be a minimizer of $\G$ among connected sets of area $\pi$ and barycenter in the origin, and let $\Lambda$ be a constant to be specified later. Assume that $E$ is not a $\Lambda$-minimizer at some point $x\in\partial^* E$. Then, there exist a sequence of radii $r_j\to 0$ and of sets $F_j$ such that for any $j$ the symmetric difference $E\Delta F_j$ is compactly contained in $B(x,r_j)$, and moreover
\begin{equation}\label{assumption}
P(E) > P(F_j) + \Lambda r_j^2\,.
\end{equation}
We claim that, at least for $j$ large enough, it is admissible to think that $F_j$ is connected. Indeed, since $x\in\partial^* E$, as soon as $j$ is large enough there exists a radius $r_j\leq r_j'\leq 2r_j$ such that the intersection of $E$ with the circle $\partial B(x,r_j')$ is exactly an arc of the circle (with length almost exactly half of the total length of the circle). As a consequence, since $E$ is connected then so is also $E\setminus B(x,r_j')=F_j\setminus B(x,r_j')$. Therefore, if we call $F_j'$ the connected component of $F_j$ containing $F_j\setminus B(x,r_j')$, then we have that $E\Delta F_j'$ is compactly contained in $B(x,r_j')$, and since $P(F_j')\leq P(F_j)$ then~(\ref{assumption}) holds also with $F_j'$ and $r_j'$ in place of $F_j$ and $r_j$, up to change $\Lambda$ with $\Lambda/4$. We have then proved our claim.\par

We want to build competitors for $\G$ starting from the sets $F_j$. Fix $j\in\N$ and denote, for the sake of readability, $F=F_j$ and $r=r_j$. Then, we define $\widetilde E$ the set obtained from $F$ with a translation and a rescaling so as to have area $\pi$ and barycenter in the origin. In other words,
\[
\widetilde E = \sigma \big( F - v\big)\,,
\]
where
\begin{align*}
\sigma = \sqrt{\frac \pi{|F|} }\,, && v = \Bar F\,.
\end{align*}
In particular, since $F$ is connected than the set $\widetilde E$ is a competitor for $\G$, and by minimality of $E$ we have $\G(\widetilde E)\geq \G(E)$. Our goal is to find a contradiction with this inequality as soon as $r$ is small enough and $\Lambda$ is large enough. To estimate $\G(\widetilde E)$, we need an upper bound for $\delta(\widetilde E)$ and a lower bound for $\lambda_0(\widetilde E)$. First of all, we can easily give bounds to $\sigma$ and $v$; indeed, since $|E|=\pi$ and $E\Delta F$ is compactly contained in a ball of radius $r$, then $\pi(1-r^2)\leq |F|\leq \pi (1+r^2)$, which for $r$ small enough leads to
\begin{equation}\label{estisigma}
1 - \frac{r^2} 2 \leq \sigma \leq 1 + r^2\,.
\end{equation}
In an analogous way, to estimate $|v|$, we can let $\overline R$ be any number larger than the diameter of $E$, and notice that, as soon as $r$ is small enough (and then the ball $B(x,r)$ has distance less than $\overline R$ from the origin), since $\Bar E=0$ we easily have
\[
|F| \Bar F = \int_F z\, dz = \int_{F\setminus E} z\, dz - \int_{E\setminus F} z\,dz\,,
\]
and then
\[
|F| |v| \leq \pi r^2 \overline R\,,
\]
which, recalling again that $|F|\geq \pi (1-r^2)$, gives
\begin{equation}\label{estiv}
|v| \leq 2 r^2 \overline R\,.
\end{equation}
To estimate $\delta(\widetilde E)$, we just have to use that $P(\widetilde E) = \sigma P(F)$, so that by~(\ref{assumption}) and~(\ref{estisigma}) we get
				\begin{equation}\label{too}
					\delta(\widetilde E) <\frac{\sigma P(E)-\sigma \Lambda r^2-2\pi}{2\pi}=\sigma \delta(E)+\sigma-1-\frac{\sigma \Lambda}{2\pi}\, r^2
					\leq \delta(E)+ \bigg(\delta(E) +1-\frac \Lambda{4\pi}\bigg)\, r^2\,,
				\end{equation}
				where we also used~(\ref{estisigma}) and the fact that $r$ is small.\par
				
				To estimate $\lambda_0(\widetilde E)$, instead, we first notice that, since $F\Delta E$ is compactly contained in a ball of radius $r$, we have
				\[
				\big| F \Delta B\big| \geq \big| E\Delta B\big| - \big|E \Delta F\big| \geq \pi \lambda_0(E) - \pi r^2\,,
				\]
				then we notice that by the definition of $\sigma$ and~(\ref{estiv}) we have
				\[\begin{split}
					\big| (F-v) \Delta B_{1/\sigma}\big| &= \big| F \Delta (B_{1/\sigma}+v)\big| \geq \big| F\Delta B\big| - \big| B\Delta B_{1/\sigma}\big|- \big| B_{1/\sigma}\Delta (B_{1/\sigma}+v)\big|\\
					&\geq \pi \lambda_0(E) - \pi r^2-\pi \bigg|\frac 1{\sigma^2}-1\bigg|-\frac 4\sigma\, |v|
					\geq \pi \Big(\lambda_0(E) - (2+4\overline R)r^2\Big)
				\end{split}\]
				and then finally that
				\[\begin{split}
					\lambda_0(\widetilde E) &= \frac{\big|(\sigma(F-v))\Delta B\big|}\pi
					=\sigma^2\, \frac{\big|(F-v)\Delta B_{1/\sigma}\big|}\pi
					\geq (1-r^2) \Big(\lambda_0(E) - (2+4\overline R)r^2\Big)\\
					&\geq \lambda_0(E)- r^2 \big( \lambda_0(E) + 2+4\overline R\,\big)\,.
				\end{split}\]
				Since this implies
				\[
				\frac 1{\lambda_0^2(\widetilde E)} \leq \frac 1{\Big(\lambda_0(E)- r^2 \big( \lambda_0(E) + 2+4\overline R\,\big)\Big)^2}
				\leq \lambda_0^{-2}(E) + 3 r^2 \big( \lambda_0(E) + 2+4\overline R\,\big)\lambda_0(E)^{-3}
				\]
				as soon as $r$ is small enough, putting this last estimate together with~(\ref{too}) we get
				\[\begin{split}
					\G(\widetilde E) &= \frac{\delta(\widetilde E)}{\lambda_0^2(\widetilde E)}
					\leq \bigg(\delta(E)+ \bigg(\bal\delta(E) +1-\frac \Lambda{4\pi}\eal\bigg)\, r^2\bigg) \, \Big(\lambda_0^{-2}(E) + 3 r^2 \big( \lambda_0(E) + 2+4\overline R\big)\lambda_0(E)^{-3}\Big)\\
					&\leq \G(E) + r^2 \bigg( \frac{\delta(E) +1-\frac \Lambda{4\pi}}{\lambda_0^2(E)}+ \frac{3\delta(E) \big(\lambda_0(E) + 2+4\overline R\big)}{\lambda_0^3(E)}\bigg)\,,
				\end{split}\]
				where the last inequality is true as soon as we have chosen $\Lambda> 4\pi(\delta(E)+1)$. Keeping in mind that $\G(E)\leq \G(\widetilde E)$ by minimality, we find the required contradiction as soon as we have chosen
				\[
				\Lambda> 4\pi \bigg(4 \delta(E) +1 + \frac{\delta(E)}{\lambda_0(E) } \big(6+12\overline R\big)\bigg)\,.
				\]
				This finally concludes the proof.
			\end{proof}
			
			The $\Lambda$-minimality given by the above lemma allows us to use classical regularity arguments, so deducing that the reduced boundary of any minimizer $E$ is ${\rm C}^{1,1}$. Keep in mind that, since we are interested in the Minkowski perimeter, the length of the reduced boundary can be less than the perimeter; more precisely, the perimeter coincides with the length of the reduced boundary plus twice the length of $\partial E\setminus \partial^* E$. However, it is straightforward that, if $E$ is a minimizer, then necessarily the set $\partial E\setminus \partial^* E$ is a union of (possibly degenerate) segments, whose endpoints meet $\partial^* E$ tangentially. Therefore, each of the paths $\tau_i$ defining $E$ is a ${\rm C}^{1,1}$ curve. The regularity of the reduced boundary implies in particular that we can consider the curvature at all simple points of the boundary (at least in a distributional sense, in principle), and then we can proceed to the proof of Theorem~\ref{formulaabc}. 

Before doing so, we briefly explain what we will mean by ``smoothly adding'' mass around a simple point of the boundary, which is something widely used in this kind of constructions. Let $F$ be a subset of $\R^2$, smooth in a small neighborhood of a simple point $P\in \partial F$. Up to a rotation, in a small ball of radius $r$ around $P$ the boundary of $F$ coincides with the graph of some regular function $\psi: I \to \R$, for a suitable interval $I\subseteq \R$. Let us now fix an arbitrary function $u \in {\rm C}^\infty_0(I)$ such that $\int_I u =1$; by saying that we ``smoothly add'' a mass $\eps$ (with $|\eps|\ll 1$) to $F$ we will then mean that we consider the set $F_\eps$ which coincides with $F$ outside of the ball $B(P,r)$, and whose boundary in $B(P,r)$ coincides with the graph of $\psi+\eps u$, see Figure~\ref{fig: smooth add mass}. It is easy but important to notice that, regardless of the particular choice of the function $u$, the following formulas for the area and perimeter of $F_\eps$ hold,
			\begin{align}\label{simfor}
				|F_\eps| = |F| + \eps\,, && P(F_\eps) = P(F) + \eps \kappa(P) + o(\eps)\,,
			\end{align}
			where $\kappa(P)$ is the curvature of $\partial F$ at $P$.
			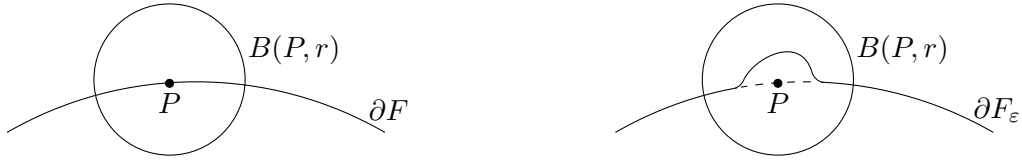
\begin{figure}[!h]
				\begin{center}
					\begin{tikzpicture}
						\draw (0,0) circle (1);
						\filldraw(0,-0.05) circle (1.5pt)node[below] {$P$};
						\draw(2.85,-0.7) arc[start angle=60, end angle=120, radius=5];
						\node at (2.9,-0.4) {$\partial F$};
						\node at (1.65,0.4) {$B(P,r)$};
						%\draw (0.5, -0.05) .. controls (0.4, 0.7) and (-0.3, 0.5) .. (-0.5, -0.1);
					\end{tikzpicture}
					\hspace{1.5cm}
					\begin{tikzpicture}
						\draw (0,0) circle (1);
						\filldraw(0,-0.05) circle (1.5pt)node[below] {$P$};
						\begin{scope}
							\clip (0,0) circle (0.5);
							\draw[dashed](2.85,-0.7) arc[start angle=60, end angle=120, radius=5];
						\end{scope}
						\begin{scope}
							\clip(2.85,-0.75) rectangle (0.58,0.05); 
							\draw(2.85,-0.7) arc[start angle=60, end angle=120, radius=5];
						\end{scope}
						\begin{scope}
							\clip(-3,-0.7) rectangle (-0.55,0.05); 
							\draw(2.85,-0.7) arc[start angle=60, end angle=120, radius=5];
						\end{scope}
						\node at (2.9,-0.4) {$\partial F_\varepsilon$};
						\node at (1.65,0.4) {$B(P,r)$};
						\draw (0.45, 0.1) .. controls (0.3, 0.6) and (-0.3, 0.3) .. (-0.45, 0);
						\draw (0.45, 0.1) arc[start angle=195, end angle=260, radius=0.193];
						\draw (-0.45, 0) arc[start angle=340, end angle=290, radius=0.19];
					\end{tikzpicture}
				\end{center}
				\caption{A set with smoothly added (positive) mass at a point $p$ on the boundary.}
				\label{fig: smooth add mass}
			\end{figure}
			Now, we are ready to present the proof of Theorem~\ref{formulaabc}.
			\proofof{Theorem~\ref{formulaabc}}
			Consider a minimizer $E$ for the ratio $\G$. We will consider first variations of the set which preserve the area of $E$ inside and outside of $B$ as well as the barycenter of the set.\par
			
			Take four simple points in $\partial^*E\setminus \overline B$ so that any three of them are not aligned, and denote them by $p_i=(x_i,y_i)$ for $1\leq i \leq 4$. Let also $m_i\in \R$ be four constants, to be specified later. Then, for any $\eps$ with $|\eps|\ll 1$, call $E_\eps$ a set obtained by smoothly adding to $E$ an area $\eps m_i$ in a small ball, with radius of order $\sqrt[3]{|\eps|}$, centered at $p_i$, for every $1\leq i\leq 4$. Let us assume that the constants $\{m_i\}$ satisfy the system
			\begin{equation}\label{simsys}
				\left( \begin{array}{cccc}
					1 & 1 & 1 & 1 \\
					x_1 & x_2 & x_3 & x_4 \\
					y_1 & y_2 & y_3 & y_4
				\end{array}\right)
				\left(\begin{array}{c}
					m_1 \\ m_2 \\ m_3 \\ m_4
				\end{array}\right)
				= \left(\begin{array}{c}
					0 \\ 0 \\ 0 
				\end{array}\right)\,.
			\end{equation}
			Notice that the first line implies that $|E_\eps|=|E|=\pi$, by the first property in~(\ref{simfor}), while the second and third lines imply that the barycenter of $E_\eps$ is very close to the origin, more precisely, $|\Bar(E_\eps)|=o(\eps)$. Moreover, the second property in~(\ref{simfor}) implies that
			\[
			P(E_\eps)=P(E) + \eps \sum_{i=1}^4 \kappa(p_i) m_i + o(\eps)\,.
			\]
			Since the four points are outside of $\overline B$, as soon as $|\eps|\ll 1$ we have $\lambda_0(E_\eps)=\lambda_0(E)+o(\eps)$. Therefore, letting $\eps\to 0$, the optimality of $E$ ensures that $\sum_i \kappa(p_i) m_i=0$. Since this is true for any choice of the $m_i$'s satisfying the system~(\ref{simsys}), we deduce that the vector of the four curvatures is a linear combination of the horizontal vectors of the matrix in~(\ref{simsys}); in other words, there exist three constants $a,\, b,\, c$ such that the formula
			\begin{equation}\label{herethecur}
				\kappa(x,y) = a x + by + c
			\end{equation}
			is true if $(x,y)$ is any of the four points $p_i$. Now, since the three points $p_1,\, p_2,\, p_3$ are not aligned, the constants $a,\,b,\, c$ are actually uniquely determined by them via the above formula; so, these constants do not vary if we replace the fourth point $p_4$ by any other simple point in $\partial^* E \setminus \overline B$. In other words, the formula~(\ref{herethecur}) is true for any simple point in $\partial^* E\setminus \overline B$. 
Notice that, if there are no simple points in $\partial ^*(E)\cap B$, the formula~\eqref{formulaabc} is proved, up to a rotation.
Hence, we can assume that there exist at least one simple point inside the barycentric ball. 
Now, we can perform the very same variation, this time choosing four simple points $p_i$, for $1\leq i\leq 4$, in $\partial^*E$ such that any three of them are not aligned and such that $p_1,p_2,p_3\notin \overline{B}$, while $p_4\in B$.
Hence, we can smoothly add (or remove) an area $\eps$ in a small ball centered at $p_4$ and an area $\eps m_i$ in a small ball centered at $p_i$, for every $1\leq i\leq 3$, where the $m_i$'s satisfy 
\begin{equation}\label{sys3+1}
				\left( \begin{array}{cccc}
					1 & 1 & 1 & 1 \\
					x_1 & x_2 & x_3 & x_4 \\
					y_1 & y_2 & y_3 & y_4
				\end{array}\right)
				\left(\begin{array}{c}
					m_1 \\ m_2 \\ m_3 \\ 1
				\end{array}\right)
				= \left(\begin{array}{c}
					0 \\ 0 \\ 0 
				\end{array}\right)\,.
			\end{equation}
As above, the first line implies that $|E_\eps|=|E|=\pi$, by the first property in~(\ref{simfor}), while the second and third lines imply that the barycenter of $E_\eps$ is very close to the origin, more precisely, $|\Bar(E_\eps)|=o(\eps)$. Moreover, the second property in~(\ref{simfor}) implies that
			\[
			P(E_\eps)=P(E) + \eps \sum_{i=1}^3 \kappa(p_i) m_i +\eps \kappa(p_4)+ o(\eps)\,.
			\]
			Since the first three points are outside of $\overline B$, we know from~\eqref{herethecur} that $\kappa(p_i)=ax_i+by_i+c$ for every $i=1,2,3$; hence, using~\eqref{sys3+1} we have
\[
\sum_{i=1}^3 \kappa(p_i) m_i=a\sum_{i=1}^3 x_i m_i+b\sum_{i=1}^3 y_i m_i+c\sum_{i=1}^3 m_i=-ax_4-by_4-c.
\]
So, 
\[
			P(E_\eps)=P(E) +\eps \big(\kappa(p_4)-ax_4-by_4-c\big)+ o(\eps)\,,
			\]
independently of the choice of the three points $p_1,p_2,p_3$.
Moreover, since $p_4$ is inside $B$, as soon as $|\eps|\ll 1$ we have 
\[
\lambda_0(E_\eps)=\lambda_0(E)-\frac{2\eps}{\pi}+o(\eps).
\]
Therefore, we can compute $\G(E_\eps)$ as 
\[
\begin{aligned}
\G(E_\eps)&=\frac{\delta(E_\eps)}{\lambda_0^2(E_\eps)}
=\frac{\delta(E)}{\lambda_0^2(E)}\left(1+\frac{\kappa(p_4)-ax_4-by_4-c}{2\pi\delta(E)}\,\eps\right)\left(1+\frac{4}{\pi\lambda_0(E)}\,\eps\right)+o(\eps)
\\&=\frac{\delta(E)}{\lambda_0^2(E)}\left(1+\frac{\kappa(p_4)-ax_4-by_4-c}{2\pi\delta(E)}\,\eps+\frac{4}{\pi\lambda_0(E)}\,\eps\right)+o(\eps).
\end{aligned}
\]
Letting $\eps\to 0$, the optimality of $E$ ensures that
\[
\frac{\kappa(p_4)-ax_4-by_4-c}{2\pi\delta(E)}+\frac{4}{\pi\lambda_0(E)}=0,
\]
that is,
\begin{equation}
\kappa(p_4)=ax_4+by_4+c-\frac{8\delta(E)}{\lambda_0(E)},
\end{equation}
thus concluding the proof, up to a rotation.
\end{proof}

\begin{remark}\label{ex24}
We can reduce ourselves to the case when the two constants $\alpha,\, \gamma$ are nonnegative. Indeed, $\alpha \geq 0$ is true up to a reflection about the $y$ axis, while the positivity of $\gamma$ is given by~\eqref{formulagamma}.
\end{remark}

\begin{remark}
The formula~\eqref{formulaabc} found for the curvature now ensures that the reduced boundary $\partial^* E$ is ${\rm C}^{1,1}$ everywhere and ${\rm C}^{\infty}$ out of $\partial B$.
\end{remark}

\section{Existence of a minimizer\label{secexist}}
This section is devoted to show the first part of Theorem~\mref{exist}, that is, the existence of a set $E$ which minimizes the ratio $\G(E)=\delta(E)/\lambda_0(E)^2$. It is not difficult to find a minimizing sequence $\{E_j\}$ which converges to a set $E$ in the $L^1$ sense; as a consequence, by the lower semi-continuity of the perimeter we directly get that $\delta(E)\leq \liminf \delta(E_j)$, while $\lambda_0(E) = \lim \lambda_0(E_j)$. Therefore, the limit set $E$ is a minimizer of the ratio, unless $E$ is a ball, in which case the ratio $\G$ cannot be defined. In fact, the difficult part of the proof is precisely in showing that the limit set cannot be the ball itself.

In particular, up to now the result of Theorem~\ref{formulaabc} might in principle be void, since it is stated for a minimizer of $\G$ and we have still not proved the existence of such minimizers. In fact, we will have to consider an auxiliary problem for which existence is guaranteed. More precisely, for a fixed constant $d>0$ we introduce the class
\[
\A_d=\Big\{E\subseteq \R^2,\, \hbox{$E$ is connected},\,|E|=\pi,\, d_H^+(E,B_E)\geq d\Big\}
\]
of the connected sets of area $\pi$ and with ``external'' Hausdorff distance from $B_E$ larger than $d$, where $B_E$ denotes the unit ball centered at the barycenter of $E$, that is, $\max \{|p-\Bar{(E)}|\,:\, p\in \partial E\}\geq 1+d$.

The first result of this section is an estimate about the length of the connected components of the boundaries of sets with bounded $\G(E)$. For brevity of notation, we will denote by $\RR$ the class of ${\rm C}^1$, connected sets with area $\pi$ which are not circles (thus, for which $\G$ is well defined).

\begin{lemma}\label{smcoco}
There exist two constants $\bar\eta<2\pi$ and $N_{\rm max}\in\N$ with the following property. Given a set $E\in \RR$, there exists another set $E'\in \RR$ with $\G(E')\leq \G(E)$, such that $\partial E'$ consists of at most $N_{\rm max}$ connected components, and all of them except at most one have length at least $\bar\eta$. Moreover, the inequality is strict if $E'\neq E$ and $\G(E)<1$. Finally, if $\G(E)<1$, then $d_H^+(E',B_{E'})\geq d_H^+(E,B_E)$.
\end{lemma}
\begin{proof} 
It is possible to fix a smooth, connected set $D\in\RR$, with connected boundary and such that $\G(D)=1$. 
Indeed, a simple calculation ensures that the ratio $\G$ of a rectangle with side-lengths $2\pi$ and $1/2$ is around $0.62$, while the ratio $\G$ of a square of area $\pi$ amounts to approximately $3.9$. Let $\bar\eta<2\pi$ be a small constant, to be specified later, and let $E\in \RR$ be any set. We want to define a set $E'\in\RR$ such that
\begin{equation}\label{almostall}
\begin{array}{c}
\G(E')\leq \G(E),\, \qquad d_H^+(E',B_{E'})\geq d_H^+(E,B_E) \hbox{ if $ \G(E)< 1$}\,, \\
\hbox{$\partial E'$ has at most one connected component shorter than $\bar\eta$}\,.
\end{array}
\end{equation}
First of all, if $\G(E)\geq 1$, then we simply take $E'=D$. In this case, (\ref{almostall}) is clearly valid (keep in mind that the boundary of $D$ is made of a single connected component). Instead, if $\G(E)<1$ and $\partial E$ has at most one connected component shorter than $\bar\eta$, then we take $E'=E$. Also in this case, the validity of~(\ref{almostall}) is obvious.\par

We want now to define a set $E'$ satisfying~(\ref{almostall}) in the case when $\G(E)<1$ and $\partial E$ has at least two connected components shorter than $\bar\eta$. In this case, assuming just for simplicity of notation that $\Bar (E)=0$, we let $P$ be a point in $\partial E$ such that $|P|-1$ is maximum, and then $d_H^+(E,B_E)=|P|-1$. Now, we let $\tau$ be a connected component of $\partial E$ with $\H^1(\tau)=\eta\leq \bar\eta$, and such that $P\notin \tau$; notice that this is possible since $P$ belongs only to one connected component of $\partial E$, and there are at least two ones which are shorter than $\bar\eta$. By regularity of $E$, we know that $\tau$ is a regular, closed curve in $\R^2$, hence $\tau=\partial G$ for some bounded set $G$. Since $E$ is connected, in particular $G$ is a ``hole'' in $E$, that is, $G\subseteq \R^2\setminus E$. By the isoperimetric inequality, we have $\eta = \H^1(\tau) \geq 2 \sqrt\pi \sqrt{|G|}$, so that
\begin{equation}\label{esti1}
|G| \leq \frac{\eta^2}{4\pi}\,.
\end{equation}
We write now $F_0=E\cup G$, and we define a slightly modified set $F$ as follows. We let $P'$ be the point in the line passing through $O$ and $P$, having distance $\eta/5$ from $P$, and such that $|P'|-1>|P|-1$. Then, we define $F\supseteq F_0$ modifying in a smooth way $\partial F_0$ near the point $P$ so to pass through $P'$. The situation is depicted in Figure~\ref{DefF}, where $F$ has an ``outer cusp'' near $P$. Notice that this construction makes sense thanks to the definition of $P$, and since by construction $P\in \partial F_0$. By taking a sufficiently small ``cusp'', so that $\partial F\setminus \partial F_0$ is very close to two copies of the segment $PP'$, we can have
\begin{align}\label{estiinter}
\big| F \setminus F_0 \big| \leq \frac{\eta^2}{12\pi}\,, && \big|\Bar(F)-\Bar(F_0)\big|\leq \frac{\eta^2}{24\pi^2}\, P(E)\,, && \H^1(\partial F) \leq \H^1(\partial F_0) + \frac \eta 2\,.
\end{align}
As a consequence, by construction and~(\ref{esti1}) we have
\begin{align}\label{estipervol}
\pi \leq |F| \leq \pi + \frac{\eta^2}{3\pi}\,, && P(F) \leq P(E) -\frac \eta 2\,.
\end{align}
\begin{figure}[thbp]
\begin{tikzpicture}[>=>>>]
\fill[green!30] (-2.5,2) .. controls (-2,2) and (-2.5,.5) .. (-1,0.5) .. controls (1.5,0.5) and (2,3) .. (2,1) .. controls (2,0.2) and (2.6,0) .. (1.8,-1) .. controls (1,-2) and (0,-1.5) .. (-1.5,-1) .. controls (-3,-.5) and (-2,.5) .. (-2.5,1) .. controls (-3,1.5) and (-3,2) .. (-2.5,2);
\draw (0,0) circle (2);
\fill (-2.8,1.9) circle (2pt);
\draw (-2.8,1.9) node[anchor=north east] {$P$};
\draw (1.1,2.4) node[anchor=north east] {$B$};
\draw (2.4,-.8) node[anchor=north east] {$E$};
\draw[line width=.8pt] (-2.5,2) .. controls (-2,2) and (-2.5,.5) .. (-1,0.5) .. controls (1.5,0.5) and (2,3) .. (2,1) .. controls (2,0.2) and (2.6,0) .. (1.8,-1) .. controls (1,-2) and (0,-1.5) .. (-1.5,-1) .. controls (-3,-.5) and (-2,.5) .. (-2.5,1) .. controls (-3,1.5) and (-3,2) .. (-2.5,2);
\fill[white] (0.5,-0.3) .. controls (1,-0.3) and (.9,.1) .. (.9,.2) .. controls (.9,.3) and (.8,.4) .. (.7,.4) .. controls (.6,.4) and (.1,.1) .. (0,0) .. controls (-.1,-.1) and (0,-.3) .. (0.5,-0.3);
\draw[line width=.8pt] (0.5,-0.3) .. controls (1,-0.3) and (.9,.1) .. (.9,.2) .. controls (.9,.3) and (.8,.4) .. (.7,.4) .. controls (.6,.4) and (.1,.1) .. (0,0) .. controls (-.1,-.1) and (0,-.3) .. (0.5,-0.3);
\draw (1.05,-.3) node[anchor=north east] {$G$};
\fill[white] (-1.3,-0.4) .. controls (-1.1,-0.4) and (-.8,-.2) .. (-.8,-.1) .. controls (-.8,0) and (-.8,.1) .. (-1.1,.1) .. controls (-1.3,.1) and (-1.5,.1) .. (-1.5,-.1) .. controls (-1.5,-.3) and (-1.5,-.4) .. (-1.3,-0.4);
\draw[line width=.8pt] (-1.3,-0.4) .. controls (-1.1,-0.4) and (-.8,-.2) .. (-.8,-.1) .. controls (-.8,0) and (-.8,.1) .. (-1.1,.1) .. controls (-1.3,.1) and (-1.5,.1) .. (-1.5,-.1) .. controls (-1.5,-.3) and (-1.5,-.4) .. (-1.3,-0.4);
%%%%
\fill[blue!30] (5.5,2) .. controls (6,2) and (5.5,.5) .. (7,0.5) .. controls (9.5,0.5) and (10,3) .. (10,1) .. controls (10,0.2) and (10.6,0) .. (9.8,-1) .. controls (9,-2) and (8,-1.5) .. (6.5,-1) .. controls (5,-.5) and (6,.5) .. (5.5,1) .. controls (5,1.5) and (5,1.7) .. (5.1,1.8) .. controls (5.2,1.9) and (4.3,2.4) .. (4.35,2.45) .. controls (4.4,2.5) and (5.2,1.9) .. (5.3,2) .. controls (5.4,2) and (5.3,2) .. (5.5,2);
\draw[<->] (5.4,2.2) -- (4.6,2.8);
\draw (5,2.5) node[anchor=south west] {$\eta/5$};
\draw (8,0) circle (2);
\fill (4.35,2.45) circle (2pt);
\draw (4.35,2.45) node[anchor=north east] {$P'$};
\draw (9.1,2.4) node[anchor=north east] {$B$};
\draw (10.4,-.8) node[anchor=north east] {$F$};
\draw[line width=.8pt] (5.5,2) .. controls (6,2) and (5.5,.5) .. (7,0.5) .. controls (9.5,0.5) and (10,3) .. (10,1) .. controls (10,0.2) and (10.6,0) .. (9.8,-1) .. controls (9,-2) and (8,-1.5) .. (6.5,-1) .. controls (5,-.5) and (6,.5) .. (5.5,1) .. controls (5,1.5) and (5,1.7) .. (5.1,1.8) .. controls (5.2,1.9) and (4.3,2.4) .. (4.35,2.45) .. controls (4.4,2.5) and (5.2,1.9) .. (5.3,2) .. controls (5.4,2) and (5.3,2) .. (5.5,2);
\draw[line width=.8pt, dashed] (8.5,-0.3) .. controls (9,-0.3) and (8.9,.1) .. (8.9,.2) .. controls (8.9,.3) and (8.8,.4) .. (8.7,.4) .. controls (8.6,.4) and (8.1,.1) .. (8,0) .. controls (7.9,-.1) and (8,-.3) .. (8.5,-0.3);
\draw (9.05,-.3) node[anchor=north east] {$G$};
\fill[white] (6.7,-0.4) .. controls (6.9,-0.4) and (7.2,-.2) .. (7.2,-.1) .. controls (7.2,0) and (7.2,.1) .. (6.9,.1) .. controls (6.7,.1) and (6.5,.1) .. (6.5,-.1) .. controls (6.5,-.3) and (6.5,-.4) .. (6.7,-0.4);
\draw[line width=.8pt] (6.7,-0.4) .. controls (6.9,-0.4) and (7.2,-.2) .. (7.2,-.1) .. controls (7.2,0) and (7.2,.1) .. (6.9,.1) .. controls (6.7,.1) and (6.5,.1) .. (6.5,-.1) .. controls (6.5,-.3) and (6.5,-.4) .. (6.7,-0.4);
\end{tikzpicture}
\caption{Construction of the set $F$ starting from $E$ in Lemma~\ref{smcoco}.}
\label{DefF}
\end{figure}
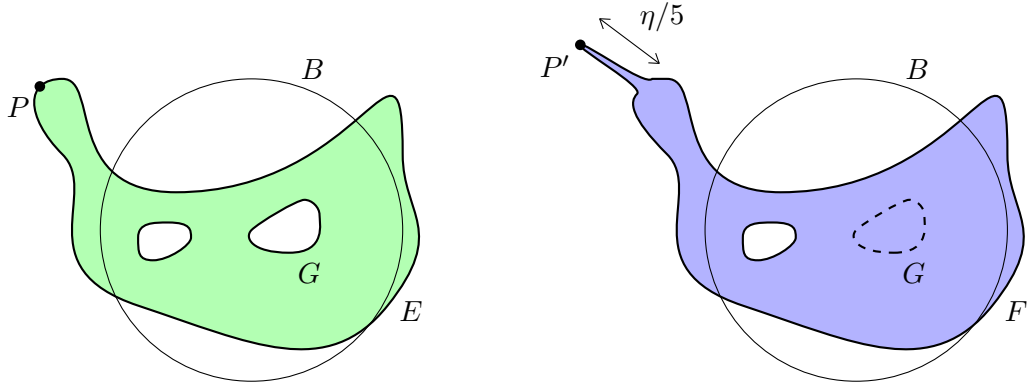
We observe that $\sqrt{\pi/|F|} F$ belongs to $\RR$. Indeed, the set $F$ is by construction ${\rm C}^1$, connected and its rescaled has area $\pi$, and moreover since $\partial E$ has at least two short connected components, then $\partial F$ has at least one, so $F$ cannot be a circle. We claim now that, provided $\bar\eta$ has been chosen small enough,
\begin{equation}\label{claimend}
\G(F)\leq \G(E)\,.
\end{equation}
Let us call again $B_E$ and $B_F$ the barycentric balls of $E$ and $F$ respectively. Namely, $B_E$ is a ball of volume $|E|=\pi$ centered at $\Bar(E)$, while $B_F$ is a ball of volume $|F|$ centered at $\Bar(F)$. Keeping in mind that the diameter of $E$ is at most $P(E)/2$, by~(\ref{esti1}) we have
\[
\pi \big|\Bar(F_0)\big| \leq |F_0|\, \big|\Bar(F_0)\big|=\bigg| \int_{F_0} z\,dz\bigg|=\bigg| \int_G z\,dz\bigg|\leq |G| \, \frac{P(E)}2 \leq \frac{\eta^2}{8\pi}\, P(E)\,,
\]
which by~(\ref{estiinter}) implies
\begin{equation}\label{movbar}
|\Bar(F)| \leq |\Bar(F_0)| + \frac{\eta^2}{24\pi^2}\, P(E)\leq \frac{\eta^2}{6\pi^2}\, P(E)\,
\end{equation}
As a consequence, recalling~(\ref{estipervol}) we have
\[
\big| B_E \Delta B_F\big| \leq 4\big|\Bar(E)-\Bar(F)\big| + |F|-\pi \leq \bigg(\frac 2{3\pi^2}\, P(E) + \frac 1{3\pi}\bigg)\, \eta^2\,.
\]
We deduce
\[
\lambda_0(F) |F| = \big| F \Delta B_F\big| \geq |E\Delta B_E| - |F\Delta E| - |B_E\Delta B_F| \geq \pi \lambda_0(E) -\bigg(\frac 2{3\pi^2}\, P(E) + \frac 2{3\pi}\bigg)\, \eta^2\,.
\]
An immediate check, using again~(\ref{estipervol}),
\[
\lambda_0(F) \geq \lambda_0(E) - \bigg( \frac{\lambda_0(E)}{3\pi^2}+\frac {2P(E)}{3\pi^3} + \frac 2{3\pi^2}\bigg)\, \eta^2\,.
\]
Since by definition $P(E)=2\pi (\delta(E)+1)$, while $\lambda_0\leq 2$ for any set, the above estimate implies
\begin{equation}\label{estilambda}
\lambda_0(F) \geq \lambda_0(E) - \bigg(\frac {4\delta(E)+8}{3\pi^2} \bigg)\, \eta^2\,.
\end{equation}
Moreover, (\ref{estipervol}) implies
\[
2\pi \big(\delta(F) +1\big) \leq 2 \sqrt{\pi|F|}\, \big(\delta(F) +1\big) =P(F) \leq P(E)-\frac \eta 2 =2\pi\big(\delta(E)+1\big)-\frac\eta 2\,,
\]
which implies
\begin{equation}\label{estidelta}
\delta(F)\leq \delta(E)-\frac\eta{4\pi}\,.
\end{equation}
Let us now check the validity of the inequality~(\ref{claimend}), namely, $\G(F)\leq \G(E)$. Using~(\ref{estilambda}) and~(\ref{estidelta}), the inequality is proved if
\[
\frac{\bal\delta(E)-\frac \eta{4\pi}\eal}{\bal\bigg(\lambda_0(E)- \bigg(\frac {4\delta(E)+8}{3\pi^2} \bigg)\, \eta^2\bigg)^2\eal} \leq \frac{\delta(E)}{\lambda_0(E)^2}\,,
\]
which in turn is true if
\begin{equation}\label{lastone}
\lambda_0(E) \geq 8 \delta(E)\, \frac {4\delta(E)+8}{3\pi}\,\eta\,.
\end{equation}
Keep in mind that we are assuming that $\G(E)<1$, which means $\delta(E)< \lambda_0(E)^2$; since $\lambda_0(E)\leq 2$, this gives both $\delta(E)\leq 2\lambda_0(E)$ and $\delta(E)\leq 4$. Therefore,
\[
8\,\delta(E)\, \frac {4\delta(E)+8}{3\pi}\,\eta \,\leq\, 16\,\lambda_0(E)\, \frac{24}{3\pi}\,\eta =\frac{128}{\pi}\,\lambda_0(E)\,\eta\,,
\]
so the validity of~(\ref{lastone}) is surely obtained if
\[
1 \geq \frac {128}{\pi}\,\eta\,.
\]
Thus, we have proved that~(\ref{claimend}) holds true as soon as $\bar\eta$ has been chosen small enough. Finally, writing for brevity $F'=\sqrt{\pi/|F|} F$, we claim that
\begin{equation}\label{claimHaus}
d_H^+(F',B_{F'}) \geq d_H^+(E,B_E)\,.
\end{equation}
To show this inequality, we first consider $d_H^+(F,B_F)$. By construction, (\ref{movbar}) and~(\ref{estipervol}), we have
\[\begin{split}
d_H^+(F,B_F)&\geq d_H^+(E,B_E) + \frac\eta 5-\big| \Bar(E) - \Bar(F)\big|- \bigg|\sqrt{\frac{|F|}\pi} -1\bigg|\\
&\geq d_H^+(E,B_E) + \frac\eta 5 - \bigg(\frac{P(E)}{6\pi^2}+\frac 1{6\pi}\bigg)\, \eta^2\,.
\end{split}\]
Then, again by~(\ref{estipervol}), keeping in mind that $d_H^+(E,B_E)\leq P(E) = 2\pi (\delta(E)+1)\leq 10\pi$, and up to further decreasing $\bar\eta$ if necessary, we have
\[\begin{split}
d_H^+(F',B_{F'}) &= \sqrt{\frac \pi{|F|}}\, d_H^+(F,B_F) \geq \bigg(1-\frac{\eta^2}{6\pi^2}\bigg) \cdot \bigg(d_H^+(E,B_E) + \frac\eta 5 - \bigg(\frac{P(E)}{6\pi^2}+\frac 1{6\pi}\bigg)\, \eta^2\bigg)\\
&\geq d_H^+(E,B_E)\,,
\end{split}\]
so that~(\ref{claimHaus}) is established. Summarizing, starting from a set $E\in\RR$ with $\G(E)<1$ and such that $\partial E$ has at least two connected components shorter than $\bar\eta$, we have built another set $F'\in\RR$ such that $\G(F')=\G(F)\leq \G(E)$ by~(\ref{claimend}), and $d_H^+(F',B_{F'}) \geq d_H^+(E,B_E)$ by~(\ref{claimHaus}). Moreover, $\partial F'$ has one connected component less than $\partial E$. Repeating this construction finitely many times, we end up with a set $E'$ satisfying~(\ref{almostall}). Thus, we have been able to find a set $E'$ satisfying~(\ref{almostall}) starting from any set $E\in\RR$.\par

In order to conclude the proof, we only have to provide a bound on the number of connected components of $\partial E'$. However, this is very simple; indeed, by construction we have $\G(E')\leq 1$, and then $P(E')=2\pi(\delta(E')+1)\leq 2\pi (4\G(E')+1)\leq 10\pi$, and since all the connected components of $\partial E'$ except at most one have length at least $\bar\eta$, then their number is at most
\[
1+\frac{10\pi}{\bar\eta}\,,
\]
so the thesis is concluded by taking any natural number $N_{\rm max}$ larger than this last quantity.
\end{proof}

We can now prove that the minimization of $\G$ in the class $\A_d$ has a solution. This is true for every $d$, but we limit ourselves to consider the case $d\leq2$.
\begin{prop}\label{exminAd}
For any $d\leq 2$, there exists a set $\overline E$ which minimizes $\G$ in the class $\A_d$. More in general, there exists a minimizer for $\G$ in the class $\A_d$ as soon as $\inf\{\G(E),\, E\in \A_d\} < 1$.
\end{prop}
\begin{proof}
Let us take any sequence of ${\rm C}^1$ sets $E_j\in \A_d$ such that $\G(E_j)$ converges to $\inf\{\G(E),\, E\in\A_d\}$. Without loss of generality, we can assume that $\G(E_j)<1$ for every $j$; indeed, the ratio $\G$ for a rectangle of sidelengths $2\pi$ and $1/2$ is around $0.62$, and such a rectangle belongs to $\A_2$, so to any $\A_d$ with $d\leq 2$; of course, even if $d>2$ we can assume that $\G(E_j)<1$ if by assumption $\inf\{\G(E),\, E\in \A_d\} < 1$. Thanks to Lemma~\ref{smcoco}, we can assume without loss of generality that the boundary of each set $E_j$ consists of at most $N_{\rm max}$ connected components. As a consequence, up to passing to a subsequence, we can assume that all the boundaries have exactly $N$ connected components, for some $1\leq N\leq N_{\rm max}$. Let us then call $\tau_1^j,\, \tau_2^j,\, \dots\,,\, \tau_N^j$ the $N$ connected components of $\partial E_j$, which are regular, closed paths, that we parametrize at constant speed on $\S^1$; in particular, we call $\tau_1^j$ the external part of the boundary, that is, $E_j$ coincides with the internal part of the curve $\tau_1^j$ to which we remove the internal parts of the other curves $\tau_i^j$ with $i>1$. Up to a translation, all the sets $E_j$ have barycenter in the origin; moreover, since $\G(E_j)<1$, in particular $\delta(E_j)<4$ and so $P(E_j)<10\pi$. As a consequence, all the paths $\tau_i^j$ have length at most $10\pi$ and they are all contained in a ball centered in the origin and with radius $10\pi$. Thus, up to a further subsequence, we can assume that for every $1\leq i \leq N$ the paths $\tau_i^j :\S^1\to\R^2$ are uniformly converging to limiting paths $\tau_i$. Notice that all the paths $\tau_i$ are non-crossing curves by definition, and by construction each curve $\tau_i$ with $i>1$ is contained in the internal part of $\tau_1$. As a consequence, the curves $\tau_i$ determine a connected set $\overline E$, and we claim that this set minimizes $\G$ within $\A_d$.

First of all, we can notice that $\overline E$ has area $\pi$, since so does every set $E_j$. Moreover, by lower-semicontinuity, for each $1\leq i\leq N$ we have
\[
\H^1(\tau_i)\leq \liminf_{j\to\infty} \H^1(\tau_i^j)\,,
\]
and then $P(\overline E)\leq \liminf P(E_j)$, thus $\delta(\overline E)\leq \liminf \delta(E_j)$. Since the sets $E_j$ converge to $\overline E$ in the $L^1$ sense, we also have $\lambda_0(\overline E)=\lim \lambda_0(E_j)$. Therefore
\[
\G(\overline E)\leq \liminf \G(E_j) = \inf\{\G(E),\, E\in \A_d\}\,.
\]
Then, to conclude the proof, we only have to check that $\overline E$ belongs to $\A_d$. Since the sets $E_j$ are converging in the $L^1$ sense to $\overline E$ and they are equibounded, we have $\Bar(\overline E)=0$ since all the barycenters are in the origin. Thus, we need to find a point $Q\in \partial \overline E$ with $|Q| \geq 1+d$; however, since each $E_j$ belongs to $\A_d$, there exist points $Q_j \in \partial E_j$ with $|Q_j|\geq 1+d$. Again possibly up to a subsequence, the points $Q_j$ converge to some point $Q$, which by construction belongs to $\partial\overline E$ and is as desired. The proof is then concluded.
\end{proof}

Arguing as in the proof of Theorem~\ref{formulaabc}, we have the following property.

\begin{lemma}\label{formulainAd}
Let $E$ be a minimizer of $\G$ in $\A_d$, with barycenter in the origin. Then, the curves forming $\partial E$ are of class ${\rm C}^{1,1}$ and, up to a rotation, there are three constants $\alpha,\, \beta,\,\gamma\in\R$ such that at every simple point $P=(x,y)\in\partial E\setminus \partial B$, except at most one, the curvature of $\partial E$ at $P$ is well defined, and it verifies
\begin{equation}\tag{\ref{abc}}
\kappa(x,y)=\alpha x + \beta - \gamma \Chi{B}\,.
\end{equation}
\end{lemma}
We recall that a point in $\partial E$ is called ``simple'' if it is contained in exactly one of the curves $\tau_i$ identifying $E$, and this curve passes through the point only once.
\begin{proof}
The proof is exactly the same as the one of Theorem~\ref{formulaabc}. In particular, the regularity is standard, and the formula for the curvature is obtained for all the points where a variation can be performed. As in Theorem~\ref{formulaabc}, this is not possible at a point $P\in \partial E\cap \partial B$, since we need that all the points around $P$ are outside of $B$, or all inside $B$, and this is also not possible for a point which is not simple.
		
Let then $P$ be a simple point in $\partial E\setminus \partial B$; of course it is possible to perform a variation if $|P| \neq 1+ d$, since a small variation does not modify the value of $d$. Let, instead, assume that $|P|= 1+d$. Then, it is still possible to perform a variation ``towards the outside'', since sets in $\A_d$ have Hausdorff distance with $B$ larger or equal than $d$, and so increasing this distance leaves a set in the class $\A_d$. Instead, it is in principle not possible to perform a variation ``towards the inside'', because this might decrease the Hausdorff distance with $B$ and then rule the set out of the class $\A_d$. Nevertheless, the problem occurs only if there is \emph{exactly one point} of $\partial E$ with distance $1+d$ from the origin, because, if $P$ is not the unique such point, then a variation towards the inside near a point does not change the Hausdorff distance. In other words, the validity of the formula is also true for points of $\partial E$ having distance $1+d$ from the origin, unless there is only a single such point.
\end{proof}

As a consequence of this formula, we are now able to prove the existence of a minimizer for the functional $\G$.

\begin{proof}[Theorem~\mref{exist}, existence part]
We prove now that there exists a connected set which minimizes the functional $\G$. The fact that any minimizer is simply connected will be proved in Section~\ref{minsimcon}. For simplicity, we subdivide the proof in a few steps.
		
\step{I}{The minimizing sequence $\{E_j\}$ converging to $B$.}
Let us take any sequence $\{E_j\}$ of connected, planar sets, such that $\G(E_j) \to \inf \{ \G \}$. Up to a translation, we can assume that all the sets $E_j$ have barycenter in the origin. Since the ratio $\G$ is defined for all the planar, connected sets except the ball, then each set $E_j$ is surely not the ball, thus $d_j = d_H^+(E_j,B)>0$. Since we have noticed in Proposition~\ref{prop:estim_mod_dumb} that the value of $\G$ for the modified dumbbell is around $1/4$, it is admissible to assume that $\G(E_j)<1$ for every $j$. 
By Proposition~\ref{exminAd}, then, we can assume without loss of generality that each $E_j$ minimizes $\G$ within the corresponding class $\A_{d_j}$. As a consequence, and up to rotations, Lemma~\ref{formulainAd} provides us with three sequences $\{\alpha_j\},\, \{\beta_j\},\,\{\gamma_j\}$ with the property that the curvature at each simple point $P\equiv (x,y) \in \partial E_j\setminus \partial B$, except at most one, is given by
		\begin{equation}\label{formj}
			\kappa(x,y) = \alpha_j x +\beta_j- \gamma_j \Chi{B}\,.
		\end{equation}
		By Lemma~\ref{smcoco}, up to a subsequence we can assume that the boundaries of each $E_j$ consist of exactly $N$ connected components, for some $1\leq N \leq N_{\rm max}$, and we call them $\tau^1_j,\, \tau^2_j,\, \dots\,,\, \tau^N_j$, being $\tau^1_j$ the external one. Since all the sets $E_j$ have barycenter in the origin, and since $P(E_j)<10\pi$ for each $j$ by the fact that $\G(E_j)<1$, then all the sets are contained in the ball centered in the origin and with radius $10\pi$. Therefore, up to a further subsequence we can assume that each sequence $\{\tau^i_j\}_{j\in\N}$ (parametrized at constant speed) is converging to some path $\tau^i$. In particular, the sets $E_j$ are converging to some set $E$ in the $L^1$ sense; moreover, the set $E$ is connected and its boundary consists of the paths $\tau^i$. By the lower semicontinuity of the length of the paths, we know that $P(E) \leq \liminf P(E_j)$, thus $\delta(E)\leq \liminf \delta(E_j)$, and by the $L^1$ convergence of $E_j$ to $E$ we also know that $\lambda_0(E) =\lim \lambda_0(E_j)$, since the barycenter of $E$ is also in the origin. As a consequence, we immediately have that
		\[
		\G(E) = \frac{\delta(E)}{\lambda_0(E)^2} \leq \liminf \frac{\delta(E_j)}{\lambda_0(E_j)^2} = \liminf \G(E_j) = \inf \G\,,
		\]
		and then the set $E$ is a minimizer of $\G$, unless $E=B$, in which case $\G(E)$ is not well defined. In other words, the proof is concluded as soon as we can exclude that $E=B$. We assume then that $E=B$ and we look for a contradiction.
		
		\step{II}{The boundary of each $E_j$ is a single curve and~(\ref{formj}) holds.}
		The goal of this step is to reduce ourselves to the case when the number $N$ of connected components of each $\partial E_j$ is just $1$, so we assume by contradiction that $N>1$. Let us call $E_j^+$ the set whose boundary is $\tau^1_j$, then $E_j^+$ consists of $E_j$ plus all the ``holes'', i.e., the sets whose boundary is $\tau^i_j$ for $i>1$. Since $E_j$ is converging to $B$ in the $L^1$ sense, then $\lambda_0(E_j)\to 0$, and since $\G(E_j)<1$ then also $\delta(E_j)\to 0$; this means that the perimeter of $E_j$ is converging to $2\pi$. However, since the area of $E_j^+$ is more than $\pi$, we have
		\[
		P(E_j) = \H^1(\tau^1_j) + \sum_{i=2}^N \H^1(\tau^i_j)\geq 2\pi + \sum_{i=2}^N \H^1(\tau^i_j)\,,
		\]
		which since $P(E_j)\to 2\pi$ implies that the latter sum is going to $0$. Up to a subsequence, we can then assume the sum to be smaller than $\bar\eta$ for any $j$, where $\bar\eta$ is given by Lemma~\ref{smcoco}. According to Lemma~\ref{smcoco}, we deduce that the only possible values of $N$ are $1$ and $2$, and since we are assuming $N>1$ we have then $N=2$. In other words, each set $E_j$ has exactly a hole, whose perimeter is going to $0$ (and then, of course also the area of the holes is going to $0$).
		
		We can easily exclude that $E_j^+$ is a ball. Indeed, if $E_j^+$ is a ball, since for all simple points of $\partial E_j^+\setminus \partial B$ the curvature is given by~(\ref{formj}), then we deduce that $\alpha_j$ must be $0$, and that $\beta_j$ must converge to $1$ for $j\to +\infty$. In addition, keeping in mind that
		\[
		\gamma_j = \frac{8\delta_j}{\lambda_j} = 8 \lambda_j \G(E_j)\,,
		\]
		we have also that $\gamma_j$ is converging to $0$ for $j\to +\infty$. But then, formula~(\ref{formj}) only attains values very close to $1$, in particular positive values if $j$ is large enough. And this is in contradiction with the existence of a hole, because the boundary of a hole must have some simple point where the curvature is negative, against the validity of~(\ref{formj}).
		
		Having shown that $E_j^+$ is not a ball, we can define $F_j$ as the translation of $\sqrt{\pi/|E_j^+|} E_j^+$ which has barycenter in the origin. The new sequence $F_j$ is then another sequence of sets with area $\pi$, barycenter in the origin, and which are converging to the ball; since no $F_j$ is a ball, the ratio $\G(F_j)$ is defined. Since the boundary of each $F_j$ is a single curve (a slight translation of a slight rescaling of $\tau^1_j$), we will conclude this step once we show that also $\{F_j\}$ is a minimizing sequence. This can be easily obtained with an estimate very close to the one already done in Lemma~\ref{smcoco}. Indeed, let us call $v_j = |E^+_j|-\pi$ the area of the hole of $E_j$, which has boundary $\tau^2_j$. The perimeter of $F_j$ satisfies
		\[
		P(F_j) = \sqrt{\frac \pi {\pi + v_j}} \, P(E^+_j) \leq P(E^+_j) \leq P(E_j) - \H^1(\tau^2_j) \leq P(E_j) - 2 \sqrt \pi \sqrt{v_j}\,,
		\]
		and then
		\[
		2\pi \delta(F_j) = P(F_j) - 2\pi \leq P(E_j)- 2 \sqrt \pi \sqrt{v_j} - 2\pi =2\pi \delta(E_j) - 2 \sqrt \pi \sqrt{v_j}\,,
		\]
		which gives
		\begin{equation}\label{poch}
			\delta(F_j) \leq \delta(E_j) - \sqrt{\frac{v_j}\pi}\,.
		\end{equation}
		On the other hand, since the whole $E_j$ is contained in the ball of radius $(1+d_j)$, then the distance between the barycenter of $E_j^+$ and the origin is at most $(1+d_j)v_j/\pi$. Calling $\widetilde B$ the ball with volume $\pi+v_j$ and centered in the barycenter of $E_j^+$, we have then
		\[
		\big| E_j^+ \Delta \widetilde B\big| \geq\big| E_j \Delta B\big|- \big| E_j^+ \Delta E_j\big|-\big| B\Delta \widetilde B\big|
		\geq \pi \lambda_0(E_j) - v_j - \frac{4(1+d_j)+\pi}\pi\, v_j\,.
		\]
		This implies that
		\[\begin{split}
			\pi \lambda_0(F_j) &= \frac \pi{|E_j^+|} \, \big| E_j^+ \Delta \widetilde B\big| \geq \frac{\pi}{\pi+v_j}\pi\lambda_0(E_j)- \frac{2\pi }{\pi+v_j}v_j-\frac{4(1+d_j)}{\pi(\pi+v_j)}v_j\\
&\geq
\left(1-\frac{v_j}{\pi}\right) \, \pi\lambda_0(E_j) - \frac{2 \pi }{\pi+v_j}v_j- \frac{4(1+d_j)}{\pi(\pi+v_j)}\, v_j,
		\end{split}\]
		thus
		\[
		\lambda_0(F_j) \geq \left(1-\frac{v_j}{\pi}\right)\lambda_0(E_j) - \frac{2 v_j}{\pi+v_j}-\frac{4(1+d_j)}{\pi^2(\pi+v_j)}\, v_j\,.
		\]
		Putting this estimate together with~(\ref{poch}), and keeping in mind that $\lambda_0(E_j)/\delta(E_j)\to +\infty$ since $\lambda_0(E_j)/\delta(E_j)> 1/\lambda_0(E_j)\to +\infty$, we deduce that $\G(F_j)<\G(E_j)$ for $j$ large. In other words, also $\{F_j\}$ is a minimizing sequence for $\G$. To conclude the step, then, we only have to show that formula~(\ref{formj}) is true also for $F_j$, which is not obvious since in principle $F_j$ might not be a minimizer of $\G$ in its class $\A_d$. However, the formula is true for points in $\partial E_j$, so in particular for points in $\partial E^+_j$, and since $F_j$ is a translation and rescaling of $E^+_j$ then the formula remains true, possibly with different values of $\alpha_j,\, \beta_j,\, \gamma_j$. Alternatively, we can observe the validity of the formula also noticing that we can assume $F_j$ to minimize $\G$ within the subclass of $\A_d$ (being $d=d_H^+(F_j,B)$) made by sets whose boundary is made of a single curve. Since of course the property of having the boundary made by a single curve remains true after a small variation, the proof of Lemma~\ref{formulainAd} works also for this subclass and then the formula for $F_j$ is established.
		
		\step{III}{The boundary of $E_j$ is an injective curve and~(\ref{formj}) holds.}
		In this step, we reduce ourselves to the case when $\partial E_j$ is not only a single curve, but also a simple one, i.e., there are no self-intersections. This can be obtained almost exactly as in Step~II. Indeed, assuming that $\partial E_j$ has self-intersections, we get that $\tau^1_j$ is the union of more than one simple curve. We can then call $E_j^-$ the part of $E_j$ whose boundary is the longest of the simple curves forming $\tau^1_j$. The isoperimetric inequality implies that the area of $E_j^-$ converges to $\pi$; moreover, $E_j^-$ is surely not a ball, because at any non-simple point of $\tau^1_j$ the two tangents have to coincide; as a consequence, we can again call $F_j$ the set $E_j^-$, rescaled so to have area $\pi$ and translated so to have barycenter in the origin. Calling this time $v_j=\pi - |E_j^-|$, exactly as in the previous step we have that
		\begin{align*}
			\delta(F_j) -\delta(E_j) \lesssim - \sqrt{v_j}\,, && \lambda_0(F_j) - \lambda_0(E_j) \gtrsim - v_j\,,
		\end{align*}
		and then we find again that $\G(F_j)<\G(E_j)$ for $j$ large, so that also $\{F_j\}$ is a minimizing sequence for $\G$. Again, the step is concluded up to notice that~(\ref{formj}) is valid also for $F_j$. And again, this is true because $\partial F_j$ is a rescaling and translation of a portion of $\partial E_j$; alternatively, also this time we can observe that the proof of Lemma~\ref{formulainAd} works also in the subclass of sets for which the boundary is a single, injective curve, because small variations do not rule sets out of this subclass.

		\step{IV}{Conclusion.}
		Summarizing, we have a sequence $\{E_j\}$ of sets with area $\pi$, barycenter in the origin, which are converging to the ball $B$ and for which $\G(E_j)\to \inf\{\G\}$; in addition, we know that the boundary of each $E_j$ is a single, injective curve, and that for all points of $\partial E_j\setminus \partial B$ the formula~(\ref{formj}) for the curvature holds. In addition, for almost all the points of $\partial E_j\cap\partial B$ the curvature is well-defined and equal to $1$. Since the Hausdorff distance between $E_j$ and $B$ is going to $0$, we deduce that necessarily $\alpha_j$ and $\gamma_j$ are converging to $0$, while $\beta_j$ is converging to $1$.
		
		But then, for $j$ sufficiently large the sets $E_j$ are nearly spherical, and, more than that, convex. However, the estimates of~\cite{CL,CL2,AFN} (see also the argument in~\cite{BCH}) ensure that, for such sets, one has
		\begin{equation}\label{reality}
			\lim_{j\to +\infty} \G(E_j) > \G(C)\,,
		\end{equation}
		where $C$ is a particular stadium, called ``optimal biscuit''. Since $C$ is a connected set with area $\pi$ and barycenter in the origin, this is in contradiction with the fact that $\{E_j\}$ is a sequence minimizing $\G$, and the contradiction concludes the thesis.
		It is important to remark that, in the above-cited papers, the authors did not claim the inequality~(\ref{reality}), but something different; however, what they proved was actually~(\ref{reality}). Let us be more precise. In those papers, the problem under consideration was to minimize not the ratio $\G(E)=\delta(E)/\lambda_0(E)^2$, but the ``classical'' ratio $\F(E)=\delta(E)/\lambda(E)^2$, where $\lambda(E)\leq \lambda_0(E)$ is the standard Fraenkel asymmetry; and in fact, what they claimed was not~(\ref{reality}) but
		\begin{equation}\label{reality2}
			\lim_{j\to +\infty} \F(E_j) > \F(C)\,.
		\end{equation}
		However, in their calculations, they were not estimating the Fraenkel asymmetry $\lambda(E)$, but the (much simpler to evaluate) barycentric asymmetry $\lambda_0(E)$, as it is explicitly written in~\cite[(74) of Lemma~4.2]{CL}. Thus, relying on the fact that $\lambda\leq \lambda_0$, and thus $\G\leq \F$, they obtained~(\ref{reality2}) by showing that (rewriting their estimate with our notation)
		\[
		\lim_{j\to +\infty} \F(E_j) \geq\lim_{j\to +\infty} \G(E_j) > \F(C)\,.
		\]
		Moreover, as observed in the same papers, for the set $C$ the ``Fraenkel ball'' is the barycentric one, which in our language means that $\F(C)=\G(C)$. As a consequence, their results actually give exactly~(\ref{reality}).
	\end{proof}
	
\begin{remark}
Keep in mind that, in order to complete the proof of Theorem~\mref{exist}, we still have to prove that minimizers are simply connected. One might notice that, in the above proof of the existence part, we have actually showed that the sets $E_j$ of a minimizing sequence are simply connected, and one might want to deduce from this information the simple connectedness of a minimizer. However, this argument does not work. In fact, we showed the simple connectedness of the sets $E_j$ \emph{only under the assumption} that these sets are converging to the ball, and in the end we have established that this leads to a contradiction, so we have no actual information on the real structure of the elements of a minimizing sequence $\{E_j\}$. By the way, in that proof we had also obtained the fact that the sets $E_j$ had an injective boundary, while the conjectured minimizer is the ``modified dumbbell'' of Figure~\ref{fig:dumbbell}, for which the boundary is not injective.
\end{remark}	
	
\section{Simple connectedness of minimizers\label{minsimcon}}
	
In this last section, we prove the second part of Theorem~\mref{exist}, that is, that any minimizer of $\G$ is simply connected. To achieve this result, for the whole section we consider a fixed minimizer $E$ for the functional $\G$ among the connected sets with area $\pi$ and barycenter in the origin. Moreover, we assume that $E$ is not simply connected, and in particular we let $\tau:\S^1\to\R^2$ be the boundary of a ``hole''; with the notation of Section~\ref{secdefin}, this means that $\tau$ corresponds to some $\tau_i$ with $i\geq 2$. Keep in mind that, as observed after Lemma~\ref{Lambda-min lemma}, we already know that $\tau$ is of class ${\rm C}^{1,1}$.

We will call $H$ the interior of the bounded set such that $\partial H=\tau$, and $F=E\cup H$, so that $E=F\setminus H$. We will consider the following two possible cases separately,
\begin{enumerate}[(i)]
\item $\tau$ intersects $\R^2\setminus \overline B$;
\item $\tau$ is contained in $\overline B$.
\end{enumerate}
The first case will be extremely easy to handle, while the other one will require a bit more care.
	
\subsection{The case (i)\label{casei}}

We start considering the case~(i), namely, if a part of the curve $\tau$ lies outside of $\overline B$. Keep in mind that, by Remark~\ref{ex24}, we can assume $\alpha$ and $\gamma$ to be positive. 

\begin{lemma}\label{lem:no hole outside}
The presence of a hole which is not contained in $\overline{B}$ is impossible.
\end{lemma}
\begin{proof}
Let $P=(x_P,y_P)$ be a point of $\tau$ with maximal distance from the origin, so in particular $P\notin \overline B$; the maximality implies that the (distributional) curvature $\kappa(P)$ must be smaller than $-|P|^{-1}$, hence strictly negative (and in principle, it might be $-\infty$). Let now $Q=(x_Q,y_Q)$ be a point of $\partial E$ with minimal coordinate $x_Q$. Again by minimality, the curvature $\kappa(Q)$ must be positive. Notice that if $P$ is a simple point, then the formula~(\ref{abc}) for the curvature at $P$ holds; hence, by the same formula, the curvature at any simple point not belonging to $\partial B$ and having $x$-coordinate smaller than $x_P$ is strictly negative. Since this is false at the point $Q$, this implies that $Q$ is either not simple or it belongs to $\partial B$. This is why the point $Q$ in Figure~\ref{fig:no hole outside} has been depicted on $\partial B$: however, this is just to fix the ideas, what we are going to use later is only that $\kappa_P<0$ and $\kappa_Q\geq 0$, and as noticed above this is true regardless of the fact that $P$ and/or $Q$ are simple or on $\partial B$.\par

Let us now modify the set $E$ as follows. For any $\eps>0$, we add a mass $\eps$ around $P$ and remove a mass $\eps$ from $Q$.
Now, choose two simple points $R,L$ on $\tau_1\setminus \overline{B}$, that is, two points on the external boundary which lie outside of $\overline{B}$, with $x_L<x_R$; then, add a mass $\sigma\eps$ around $L$ and remove it from $R$, choosing the constant $\sigma>0$ in such a way that the barycenter remains in the origin, as shown in Figure~\ref{fig:no hole outside}.
A clarification is in order.
Notice that we can always move the barycenter along the vertical direction with no alteration of the perimeter, since we can choose two points on the boundary with the same curvature and add mass in one point while removing the same mass from the other one. In this way, it is sufficient to choose $\sigma$ such that the barycenter has null horizontal coordinate.  
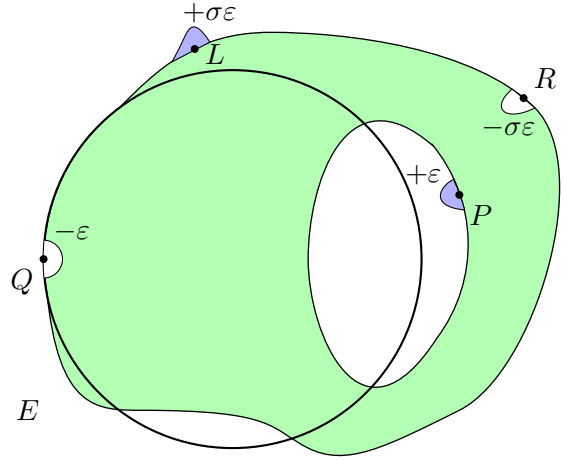
\begin{figure}[!h]

\begin{center}
\begin{tikzpicture}
\filldraw[green!30, draw=black, line width=.5] (4,2) .. controls (3.2,2.8) and (1.5,3) .. (0.5,3) .. controls (-.5,3) and (-1,2.5) .. (-1.5,2) arc (127:190:2.5) .. controls (-2.3,-2) and (-1.8,-2) .. (-1.2,-2) .. controls (0.6,-2) and (.5,-2.25) .. (1,-2.5) .. controls (1.5,-2.75) and (2,-2.5) .. (3,-2) .. controls (4,-1.5) and (4.8,1.2) .. (4,2);
\filldraw[color=white, draw=black, line width=.5] (2.65,1.5) .. controls (1.5,2.5) and (1,1) .. (1,0) .. controls (1,-1) and (1.7,-2.6) .. (2.7,-1.05) arc (-37:40:2.07);
\draw[line width=.8pt] (0,0) circle (2.5);
\begin{scope}
\clip (0,0) circle (2.5);
\filldraw[white!30, draw=black, line width=.4] (-2.5,0) circle (0.25);
\end{scope}
\fill (-2.5,0) circle (1.5pt) node[anchor=north east] {$Q$};
\draw (-2.5,0.1) node[anchor=south west] {$-\eps$};
\filldraw[blue!30, draw=black, line width=.4] (2.94,1.06) .. controls (2.77,1) and (2.56,.7) .. (3.07,.65) arc (15:25:2.4);
\fill (3,.85) circle (1.5pt) node[anchor=north west] {$P$};
\draw (2.9,0.85) node[anchor=south east] {$+\eps$};
\filldraw[white!30, draw=black, line width=.4] (3.7,2.25) .. controls (3.4,2) and (3.6,1.8) .. (4,2) -- (3.7,2.25);
\fill (3.85,2.13) circle (1.5pt) node[anchor=south west] {$R$};
\draw (3.65,2) node[anchor=north] {$-\sigma\eps$};
\filldraw[blue!30, draw=black, line width=.4] (-0.3,2.87) .. controls (-0.55,3.26) and (-0.6,3) .. (-0.7,2.8) -- (-0.8,2.61) -- (-0.3,2.87);
\fill (-0.5,2.78) circle (1.5pt);
\draw (-0.5,2.71) node[anchor= west] {$L$};
\draw (-0.3,3) node[anchor=south] {$+\sigma\eps$};					
\draw (-3,-2) node[anchor=west] {$E$};
\end{tikzpicture}
\end{center}
\caption{The modified set $E_\eps$.}
\label{fig:no hole outside}
\end{figure}

In this way, we have defined a competitor $E_\eps$, with the property that, at least if $\kappa(P)$ and $\kappa(Q)$ are real numbers,
\[
P(E_\eps)=P(E) + \Big(\kappa(P)-\kappa(Q)+\sigma\big(\kappa(L)-\kappa(R)\big)\Big) \eps + o(\eps)\,.
\]
Keep in mind that $\kappa(P)<\kappa(Q)$ as noticed above, and moreover $\kappa(L)\leq\kappa(R)$ since $x_L<x_R$ and for both the points $L$ and $R$ the formula~(\ref{abc}) holds. As a consequence, we have $P(E_\eps)<P(E)$ as soon as $\eps\ll 1$. The same conclusion holds also if $\kappa(P)$ and/or $\kappa(Q)$ are not real numbers, since in that case
\[
\lim_{\eps\to 0} \frac{P(E_\eps)-P(E)}\eps =-\infty\,.
\]
Since $E_\eps$ has the same volume as $E$, this implies that $\delta(E_\eps)<\delta(E)$ for $\eps$ small enough. Moreover, by construction $E_\eps$ has the same baricenter as $E$, and since the points $P,\,L$ and $R$ are outside $\overline B$, we get (depending on the position of $Q$)
\[
|E\Delta B| \leq |E_\eps\Delta B|\leq |E\Delta B|+2\eps\,,
\]
and thus $\lambda_0(E_\eps)\geq \lambda_0(E)$. The bounds for $\delta(E_\eps)$ and $\lambda_0(E_\eps)$ ensure that $\G(E_\eps)<\G(E)$, and this gives the desired contradiction by the minimality of $E$.
\end{proof}

\subsection{The case (ii)\label{caseii}}
We now pass to consider the case~(ii), namely, if the curve $\tau$ is entirely contained in $\overline B$. 
We will show that this second case is impossible as well. 
To do so, we will need a quite involved argument, so we will subdivide the proof in a few steps. During the proof, we will need to evaluate three quantities in a very precise way; since this will require three calculations which are elementary but quite long and boring, we have moved them into an Appendix, so to keep the main argument clean and easier to read.

\begin{lemma}\label{lem: no hole inside B}
The presence of a hole contained in the barycentric ball is impossible.
\end{lemma}
\begin{proof}
Let $H$ be the set bounded by $\tau$ and set $|H|=a\pi$, with $a\in (0,1]$.
We now divide this proof into steps, for the sake of readability.
\step{I}{An upper bound for the area $|H|$.}
We can easily estimate the deficit of the set $E$ by 
\[
\delta(E)\geq \frac{2\pi(\sqrt{a}+\sqrt{1+a})}{2\pi}-1=\sqrt{a}+\sqrt{1+a}-1.
\]
But, by Proposition~\ref{prop:estim_mod_dumb} and since $\lambda_0(E)\leq 2$ by definition, we can deduce that it must be $\delta(E)<1$ and hence $\sqrt{a}+\sqrt{1+a}-1<1$.
Thus, solving this inequality, we have that $a<9/16$.
\step{II}{A lower bound for the area $|H|$.}
Let us consider the set $F$ obtained by filling the hole $H$; namely, set $F=E\,\cup\, H$ and consider it as a competitor.
Let us start by estimating its deficit.
Since we filled the hole, we can write the perimeter of $F$ as follows
\begin{equation*}
P(F)=P(E)-P(H)=2\pi(\delta(E)+1)-P(H)\leq 2\pi(\delta(E)+1)-2\pi\sqrt{a}.
\end{equation*}
Now, using the fact that $|F|=|E|+|H|=(1+a)\pi$, we can bound $\delta(F)$ by above by
\begin{equation}\label{eq:delta of F}
\delta(F)=\frac{P(F)}{2\pi\sqrt{1+a}}-1\leq \frac{\delta(E)+1-\sqrt{a}}{\sqrt{1+a}}-1=:\tilde{\delta}.
\end{equation}
Now we need an estimate for $\lambda_0(F)$.
First, notice that the barycenter of $F$ is given by
\[
\Bar{F}=\frac{\int_Fz dz}{(1+a)\pi}=\frac{\int_Ez dz+\int_H zdz}{(1+a)\pi}=\frac{\int_Hz dz}{(1+a)\pi},
\]
and hence, since $H\subseteq B$, we have the following bound
\[
|\Bar{F}|\leq \frac{|H|}{(1+a)\pi}=\frac{a}{1+a}.
\]
Now, we denote by $B'$ the barycentric ball for $F$ (i.e. the ball centered at $\Bar{F}$ of radius $\sqrt{1+a}$) and we bound $|B'\setminus F|$ as follows
\[
|B'\setminus F|\geq|B'\setminus E|-a\pi\geq |B\setminus E|- |B\setminus B'|-a\pi. 
\]
However, an easy computation shows that 
\begin{equation}\label{firstesti2}
|B\setminus B'|\leq 2\left(1+\frac{a}{1+a}-\sqrt{1+a}\right),
\end{equation}
which plugged into the previous equality gives
\[
|B'\setminus F|\geq |B\setminus E|-a\pi-2\left(1+\frac{a}{1+a}-\sqrt{1+a}\right).
\]
Therefore, since $\lambda_0(F)=2|B'\setminus F|/|F|$, we get
\begin{equation}\label{eq:estim lambdaF}
\begin{split}
\lambda_0(F)&\geq \frac{2|B\setminus E|}{(1+a)\pi}-\frac{2a\pi}{(1+a)\pi}-\frac{4}{(1+a)\pi}\left(1+\frac{a}{1+a}-\sqrt{1+a}\right)\\
&=\frac{\lambda_0(E)}{1+a}-\frac{2a}{1+a}-\frac{4}{(1+a)\pi}\left(1+\frac{a}{1+a}-\sqrt{1+a}\right)=:\tilde\lambda.
\end{split}
\end{equation}
It is not hard to show that $\tilde\lambda\geq 0$. 
Indeed, since $\G(E)< 1/4$, we have 
\[
\lambda_0(E)^2>4\delta(E)\geq 4\big(\sqrt{1+a}+\sqrt{a}-1\big),
\]
so it follows 
\[
\tilde{\lambda}>\frac{2\sqrt{\sqrt{1+a}+\sqrt{a}-1}}{1+a}-\frac{2a}{1+a}-\frac{4}{(1+a)\pi}\left(1+\frac{a}{1+a}-\sqrt{1+a}\right),
\]
which is positive for all $a\in [0,9/16]$, as we check in Lemma~\ref{lem: lambda tilde positive} in the Appendix.

Now, since $E$ is a minimizer, we have a contradiction as soon as $\G(F)<\G(E)$ and hence as soon as 
\[ 
\frac{\tilde{\delta}}{\tilde{\lambda}^2}<\frac{\delta(E)}{\lambda_0(E)^2},
\]
thanks to~\eqref{eq:delta of F} and~\eqref{eq:estim lambdaF}.
Hence, we have a contradiction whenever
\begin{equation*}
\left(\frac{\delta(E)+1-\sqrt{a}}{\sqrt{1+a}}-1\right)\lambda_0(E)^2<\delta(E) \cdot\tilde\lambda^2,
\end{equation*}
that is, whenever
\begin{equation}\label{eq: contradiction ineq}
\delta(E)\left(\frac{\lambda_0(E)^2}{\sqrt{1+a}}-\tilde\lambda^2\right)< \lambda_0(E)^2\left(\sqrt{\frac{a}{1+a}}+1-\frac{1}{\sqrt{1+a}}\right),
\end{equation}
which, using the fact that $\tilde{\lambda}\leq \lambda_0(E)/(1+a)\leq \lambda_0(E)/\sqrt[4]{1+a}$, is equivalent to
\[
\G(E) = \frac{\delta(E)}{\lambda_0(E)^2}< \frac{\sqrt{\frac{a}{1+a}}+1-\frac{1}{\sqrt{1+a}}}{\frac{\lambda_0(E)^2}{\sqrt{1+a}}-\tilde\lambda^2}.
\]
Consider now the denominator of the right-hand side. For any value of $a\in [0,9/16]$, and recalling the definition~(\ref{eq:estim lambdaF}) of $\tilde\lambda$, we can regard it as a function of $\lambda_0(E)$, and it is immediate to observe that it is an increasing function. As a consequence, since by definition $\lambda_0(E)\leq 2$, and keeping in mind that $\G(E)<1/4$ by Proposition~\ref{prop:estim_mod_dumb}, to get a contradiction it is enough to show that
\begin{equation}\label{thistimecontrad}
\frac 14 \leq \frac{\sqrt{\frac{a}{1+a}}+1-\frac{1}{\sqrt{1+a}}}{\frac{4}{\sqrt{1+a}}-\left(\frac{2}{1+a}-\frac{2a}{1+a}-\frac{4}{(1+a)\pi}\left(1+\frac{a}{1+a}-\sqrt{1+a}\right)\right)^2}\,.
\end{equation}
And in turn, as we show in Lemma~\ref{lem: 2 vs 0.62}, this is true for any $a\leq 0.14$, so we deduce that $a>0.14$.

\step{III}{An improved estimate of $\Bar{F}$ and $|B\setminus B'|$.}
Now that we have a lower bound for $|H|$, we can improve that bound~(\ref{eq:estim lambdaF}) on $\lambda_0(E)$ by improving the estimate of $|\Bar{F}|$ and hence of $|B\setminus B'|$. In fact, if we consider the circular segment $H_{x_0} :=B\cap \{(x,y): x\geq x_0\}$, where $x_0$ is chosen so that $|H|=|H_{x_0}|$, then we have
\begin{equation*}
\bigg|\int_H z\,dz\bigg| \leq  \bigg|\int_{ H_{x_0}}z\,dz\bigg|=\frac 23\,(1-x_0^2)^{3/2}.
\end{equation*}
Observing that the area of $H_{x_0}$ is given by
\begin{equation} \label{formula}
| H_{x_0}|=\arccos x_0-x_0\sqrt{1-x_0^2}\,,
\end{equation}
we claim that the ratio
\begin{equation}\label{ratio}
\frac{\big|\int_{ H_{x_0}}zdz\big|}{|H_{x_0}|}=\frac{\frac 23\,(1-x_0^2)^{3/2}}{\arccos x_0-x_0\sqrt{1-x_0^2}}
\end{equation}
is an increasing function of $x_0$ (thus a decreasing function of $a$, since $| H_{x_0}|=|H|=a\pi$). To show the claim, it is enough to compute the derivative, so we are reduced to check that
\[
0<-2x_0\sqrt{1-x_0^2}\,\arccos x_0+2x_0^2(1-x_0^2)+\frac43\big(1-x_0^2\big)^2\,,
\]
which can be rewritten as
\[\begin{split}
0&<2\sqrt{1-x_0^2}\left(x_0^2\sqrt{1-x_0^2}-x_0\arccos x_0+\frac{2}{3}\left(1-x_0^2\right)^{3/2}\right)\\
&=2\sqrt{1-x_0^2}\cdot\bigg(\bigg|\int_{H_{x_0}}zdz\bigg|-x_0\big|H_{x_0}\big|\bigg)\,.
\end{split}\]
Thus, the validity of the claim follows since an obvious geometric argument implies that
\[
\bigg| \int_{H_{x_0}} z\, dz\bigg| \geq x_0  \, \big|H_{x_0}\big|\,.
\]
We can immediately take advantage of the claim. Indeed, since by Step~II we know that $|H_{x_0}|=|H|>0.14\pi$, and then by~(\ref{formula}) we deduce $x_0<0.61$, then we can bound the ratio in~(\ref{ratio}) by the value corresponding to $x_0=0.61$, so by substituting we get
\[
\frac{\left|\int_{H_{x_0}}zdz\right|}{|H_{x_0}|} < 0.77 \leq \frac 45 \,.
\]
And finally, this implies
\[
|\Bar{F}|=\frac{\left|\int_Hzdz\right|}{(1+a)\pi}\leq\frac45 \cdot\frac{ a\pi}{(1+a)\pi}=\frac{4a}{5(1+a)}\,.
\]
Having improved the bound for $\Bar{F}$, we can now improve also the bound for $|B\setminus B'|$. Keep in mind that $B$ is the unit ball centered at the origin, while $B'$ is a ball of radius $\sqrt{1+a}$ centered at distance $\Bar{F}$ from the origin. Hence, $|B\setminus B'| \leq |B\setminus B''|$, where $B''$ is the ball with radius $\sqrt{1+a}$ centered at the point $(-4a/5(1+a),0)$.
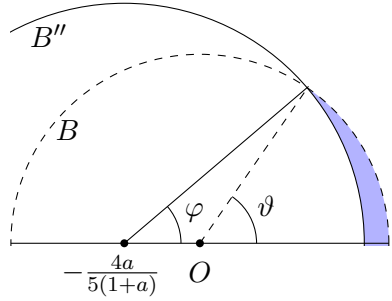
\begin{figure}[!h]
\centering
\begin{tikzpicture}[scale=2.5]
\begin{scope}
\clip (-1,-0.01) rectangle (1.2,1.5);
\fill[blue!30] (0,0) circle (1);
\fill[white] (-0.4,0) circle (1.27);
\draw[dashed] (0,0) circle (1);
\draw (-0.4,0) circle (1.27);
\end{scope}
\draw[dashed] (0,0)--(0.575,0.83);
\draw (-0.4,0)--(0.575,0.83);
\draw (0,-0.15) node {$O$};
\draw (-0.47,-0.17) node {$-\frac{4a}{5(1+a)}$};
\begin{scope}
\clip (-0.4,0)--(0.575,0.83)--(0,0)--(-0.4,0);
\draw (-0.4,0) circle (0.3);
\end{scope}
\begin{scope}
\clip (0,0)--(0.575,0.83)--(1,0)--(0,0);
\draw (0,0) circle (0.3);
\end{scope}
\draw (0.35,0.2) node {$\vartheta$};
\draw (-0.03,0.16) node {$\varphi$};
\draw (-0.7,0.6) node {$B$};
\draw (-0.8,1.1) node {$B''$};
\draw (-1,0) -- (1,0);
\fill (0,0) circle (.6pt);
\fill (-.4,0) circle (.6pt);
\end{tikzpicture}
\caption{Representation of the upper half of $B\setminus B''$.}
\label{fig: mezzaluna}
\end{figure}
We now define the angles $\vartheta$ and $\varphi$ as in Figure~\ref{fig: mezzaluna}. The angles satisfy
\begin{equation}\label{meaningphitheta}
\left\{\begin{array}{l}
\sin\vartheta=\sqrt{1+a}\sin\varphi\\
\frac{4a}{5(1+a)}=\sqrt{1+a}\cos\varphi-\cos\vartheta\,,
\end{array}\right.
\end{equation}
which readily brings to
\begin{align}\label{eq: angles theta and phi}
\vartheta=\arccos\left(\frac{25a^2+34a+25}{40(1+a)}\right)\,, && \varphi=\arccos\left(\frac{25a^2+66a+25}{40(1+a)^{3/2}}\right)\,.
\end{align}
We can now observe that the area of $B\setminus B''$ is exactly
\[
|B\setminus B''|= \vartheta+\frac{4a}{5(1+a)}\sin\vartheta-(1+a)\varphi\,,
\]
and a calculation made in Lemma~\ref{lem: ratio leq 0.62} ensures that the latter quantity is smaller than
\[
0.62 \left(1+\frac{a}{1+a}-\sqrt{1+a}\right)\,.
\]
Since $|B\setminus B'| \leq |B\setminus B''|$, we have thus obtained
\[
|B\setminus B'|\leq 0.62\cdot \left(1+\frac{a}{1+a}-\sqrt{1+a}\right)\,,
\]
which is the same estimate as in~(\ref{firstesti2}) with the coefficient $2$ replaced by $0.62$. Therefore, the same argument which was leading from~(\ref{firstesti2}) to~(\ref{eq:estim lambdaF}) is now giving $\lambda_0(F)\geq \tilde{\tilde\lambda}$ with
\begin{equation}\label{eq: new lambda tilde}
\tilde{\tilde\lambda}:=\frac{\lambda_0(E)}{1+a}-\frac{2a}{1+a}-\frac{2\cdot 0.62}{(1+a)\pi}\left(1+\frac{a}{1+a}-\sqrt{1+a}\right)\,.
\end{equation}

\step{IV}{Conclusion.}
Finally, we can argue exactly as in the end of Step~II, this time inserting~\eqref{eq: new lambda tilde} into~\eqref{eq: contradiction ineq} in place of~(\ref{eq:estim lambdaF}). Hence, in place of~(\ref{thistimecontrad}), this time we get the desired contradiction as soon as 
\[
\frac 14 \leq \frac{\sqrt{\frac{a}{1+a}}+1-\frac{1}{\sqrt{1+a}}}{\frac{4}{\sqrt{1+a}}-4\left(\frac{1}{1+a}-\frac{a}{1+a}-\frac{ 0.62}{(1+a)\pi}\left(1+\frac{a}{1+a}-\sqrt{1+a}\right)\right)^2}\,.
\]
And finally, this inequality is true for all $a\in [0,9/16]$, as we show in Lemma~\ref{lem: 2 vs 0.62}, and then the proof is completed.
\end{proof}

The simple connectedness of optimizers now directly follows from Lemma~\ref{lem:no hole outside} and Lemma~\ref{lem: no hole inside B}.

\begin{proof}[Theorem~\mref{exist}, simple connectedness part]
Let $E$ be a minimizer of $\G$ among the planar, connected sets. If $E$ is not simply connected, then one of the cases listed at the beginning of Section~\ref{minsimcon} must hold. However, we have excluded both possibilities of Case~(i) and Case~(ii) in Section~\ref{casei} and~\ref{caseii} respectively.
\end{proof}

\appendix
\section{Technical results}\label{technical appendix}
In this section, we prove the technical inequalities needed in the proof of the simple connectedness of the minimizers.
\begin{lemma}\label{lem: lambda tilde positive}
For all $a\in [0,9/16]$, it holds
\begin{equation}\label{eq: lambda tilde positive}
\sqrt{1+a}+\sqrt{a}-1\geq\left(a+\frac{2}{\pi}\left(1+\frac{a}{1+a}-\sqrt{1+a}\right)\right)^2.
\end{equation}
\end{lemma}

\begin{proof}
First, we claim that, on $[0,9/16]$,
\begin{equation}\label{horrcl}
0\leq 1+\frac{a}{1+a}-\sqrt{1+a}\leq \frac{5}{9}a.
\end{equation}
Indeed, the first inequality is equivalent to
\[
1 + \frac{2a}{1+a} + \frac{a^2}{(1+a)^2} \geq 1 + a\,,
\]
which is immediate since the third term is positive and $a \leq 1$. For the upper bound, since $\frac{a}{1+a}\leq a$ it suffices to show $1+a-\sqrt{1+a}\leq\frac59 a$, which is true exactly for $a\leq 9/16$.\par

Finally, since $a^{3/2} \leq 27/64 < 1/2$ and $\big(1+\frac{10}{9\pi}\big)^2<2$, by~(\ref{horrcl}) we get
\[
\sqrt{1+a}+\sqrt{a}-1\geq \sqrt{a} \geq 2a^2 \geq a^2\left(1+\frac{10}{9\pi}\right)^2\geq\left(a+\frac{2}{\pi}\left(1+\frac{a}{1+a}-\sqrt{1+a}\right)\right)^2\,,
\]
which is~(\ref{eq: lambda tilde positive}).
\end{proof}

\begin{lemma}\label{lem: ratio leq 0.62}
For all $a\in [0.14,\, 9/16]$, the following holds 
\[
\vartheta(a)+\frac{4a}{5(1+a)}\sin\vartheta(a)-(1+a)\varphi(a)\leq 0.62 \left(1+\frac{a}{1+a}-\sqrt{1+a}\right)\,,
\]
where the angles $\vartheta(a)$ and $\varphi(a)$ are defined as in~\eqref{eq: angles theta and phi}, that is,
\[
\vartheta(a)=\arccos\left(\frac{25a^2+34a+25}{40(1+a)}\right),\qquad \varphi(a)=\arccos\left(\frac{25a^2+66a+25}{40(1+a)^{3/2}}\right).
\]
\end{lemma}

\begin{proof}
We denote for brevity $M(a)$ the left-hand side of the claimed inequality, and
\[
L(a)=M(a)-0.62\left(1+\frac{a}{1+a}-\sqrt{1+a}\right)\,,
\]
so that the claim reduces as $L(a)<0$ for $a\in[0.14,\,9/16]$. At the left endpoint, we have
\[
L(0.14)=M(0.14)-0.62\left(1+\frac{0.14}{1.14}-\sqrt{1.14}\right)\approx 0.0337-0.0342 <0\,,
\]
so to conclude the proof it is enough to show that $L$ is decreasing on $[0.14,\,9/16]$, that is,
\begin{equation}\label{ineq: M'-0.62(..) leq 0}
M'(a)\ \leq\ \frac{\partial}{\partial a}\left(0.62+\frac{0.62a}{1+a}-0.62\sqrt{1+a}\right)
=\frac{0.62}{(1+a)^2}-\frac{0.31}{\sqrt{1+a}}=:h'(a)\,.
\end{equation}
Keeping in mind the definition~(\ref{eq: angles theta and phi}) of $\vartheta$ and $\varphi$, and using the relations~(\ref{meaningphitheta}), a lengthy but elementary calculation leads to the identity
\[
M'(a)=\frac 85\,\frac{\sin\vartheta}{(1+a)^2}-\varphi \,.
\]
As a consequence, the inequality~(\ref{ineq: M'-0.62(..) leq 0}) to prove can be rewritten as
\begin{equation}\label{ineq: phi geq Psi}
\varphi(a)\ \geq\ \Psi(a):=
\frac{\tfrac{8}{5}\sin\vartheta-0.62}{(1+a)^2}+\frac{0.31}{\sqrt{1+a}} \,.
\end{equation}
We can now observe that both sides of this inequality are decreasing for $a\in [0.14,\,9/16]$. Indeed, a direct calculation of the derivatives ensures that the terms in parentheses in~(\ref{eq: angles theta and phi}) are both increasing for $a\geq 0$, and then $\vartheta$ and $\varphi$ are decreasing, thus so is also $\sin\vartheta$ since $\vartheta \in [0,\pi/2]$ by construction. Observing also that $\frac 85\, \sin\vartheta-0.62>0$ by monotonicity and since at $a=9/16$ we have $\frac 85\, \sin\vartheta\approx 0.55$, we get that also $\Psi$ is decreasing as a sum of positive decreasing terms. To conclude the validity of~(\ref{ineq: phi geq Psi}) and thus the proof, we can then just evaluate $\varphi(a)$ and $\Psi(a)$ for a few values of $a$.\par

More precisely, we define $a_1=0.14,\, a_2=0.19,\, a_3=0.27,\, a_4=0.38,\, a_5=0.50$ and $a_6=9/16$ and we calculate $\varphi(a)$ and $\Psi(a)$ at these values. Then, for every $a_i \leq a\leq a_{i+1}$ by monotonicity we have $\varphi(a)\geq \varphi(a_{i+1})$ and $\Psi(a)\leq \Psi(a_i)$, so we only have to check that $\varphi(a_{i+1})\geq \Psi(a_i)$ for every $i \in \{1,\,2,\,3,\,4,\,5\}$. The table below collects these values (in particular, the values of $\Psi(a_i)$ are rounded up and those of $\varphi(a_i+1)$ are rounded down), and since the required inequalities are verified this concludes the thesis.
\[
\renewcommand{\arraystretch}{1.2}
\begin{array}{|c|c|c|c|}
\hline
a_i & \Psi(a_i) & a_{i+1} & \varphi(a_{i+1}) \\
\hline
0.14 & 0.7346 & 0.19 & 0.7372 \\
0.19 & 0.6750 & 0.27 & 0.6764 \\
0.27 & 0.5906 & 0.38 & 0.5953 \\
0.38 & 0.4918 & 0.50 & 0.5067 \\
0.50 & 0.4003 & 9/16 & 0.4591 \\
\hline
\end{array}
\]
\end{proof}

\begin{lemma}\label{lem: 2 vs 0.62}
For all $a\in [0,0.14]$ it holds
\[
\frac{1}{4}\leq \frac{\sqrt{\frac{a}{1+a}}+1-\frac{1}{\sqrt{1+a}}}{\frac{4}{\sqrt{1+a}}-\left(\frac{2}{1+a}-\frac{2a}{1+a}-\frac{4}{(1+a)\pi}\left(1+\frac{a}{1+a}-\sqrt{1+a}\right)\right)^2},
\]
while for all $a\in [0, 9/16]$ it holds
\[
\frac{1}{4}\leq \frac{\sqrt{\frac{a}{1+a}}+1-\frac{1}{\sqrt{1+a}}}{\frac{4}{\sqrt{1+a}}-\left(\frac{2}{1+a}-\frac{2a}{1+a}-\frac{2\cdot 0.62}{(1+a)\pi}\left(1+\frac{a}{1+a}-\sqrt{1+a}\right)\right)^2}\,.
\]
\end{lemma}
\begin{proof}
Fix $s\in\{0.62,\,2\}$ and set $\psi(a)=1+\dfrac{a}{1+a}-\sqrt{1+a}$.

\step{0}{Reduction.}
With $r_s(a)=1-a-\tfrac{s}{\pi}\psi(a)$, the squared term in the denominator equals
\[
\Big(\tfrac{2}{1+a}(1-a)-\tfrac{2s}{(1+a)\pi}\psi\Big)^2=\tfrac{4}{(1+a)^2}r_s(a)^2\,,
\]
so the denominator is
\[
D := \frac{4}{\sqrt{1+a}}-\frac{4}{(1+a)^2}r_s(a)^2=\frac{4}{(1+a)^2}\big((1+a)^{3/2}-r_s(a)^2\big)\,.
\]
Since $\psi\ge0$ on $[0,9/16]$ and $r_s>0$ (as we are going to see in Step~I), we have $r_s\le1$, hence $r_s^2\le1<(1+a)^{3/2}$ for $a>0$ and then the denominator is positive on $(0,9/16]$. 

Thus, the asserted inequality is equivalent to $4N\ge D$, where $N=\sqrt{\tfrac{a}{1+a}}+1-\tfrac{1}{\sqrt{1+a}}$ is the numerator. Since
\[
4N-D=4\sqrt{\tfrac{a}{1+a}}+4-\frac{8}{\sqrt{1+a}}+\frac{4}{(1+a)^2}r_s^2\,,
\]
$4N\ge D$ is equivalent to
\begin{equation}\label{defg}
f_s(a) := r_s^2\ \geq\ (1+a)^2\left(\frac{2}{\sqrt{1+a}}-1-\sqrt{\tfrac{a}{1+a}}\right)= (1+a)^{3/2}\big(2-\sqrt{1+a}-\sqrt a\big) =: g(a)\,.
\end{equation}
Note that $f_s(0)=g(0)=1$, and since $\sqrt{1+a}+\sqrt a$ increases from $1$ to $2$ on $[0,9/16]$, one has $g\geq 0$ with $g(9/16)=0$. We are going to prove that $f_s\geq g$ on $(0,0.14)$ for $s=2$, and on $(0,9/16)$ for $s=0.62$.

\step{I}{$f_s$ is positive, decreasing and convex, with concave derivative.}
Each derivative of $\psi$ is a difference of two powers of $1+a$, in particular
\begin{align*}
\psi'&=(1+a)^{-2}-\tfrac12(1+a)^{-1/2}\,, \\
\psi''&=-2(1+a)^{-3}+\tfrac14(1+a)^{-3/2}\,, \\
\psi'''&=6(1+a)^{-4}-\tfrac38(1+a)^{-5/2}\,.
\end{align*}
On $[0,9/16]$, we have $(1+a)^{3/2}\le(25/16)^{3/2}=125/64<2$, hence $\psi'>0$, $\psi''<0$, $\psi'''>0$.
Consequently,
\begin{align*}
r_s'=-1-\tfrac{s}{\pi}\psi'<0\,, && r_s''=-\tfrac{s}{\pi}\psi''>0\,, &&
r_s'''=-\tfrac{s}{\pi}\psi'''<0\,.
\end{align*}
In particular, $r_s$ is a decreasing function of $a\in [0,9/16]$, and then for every $s\leq 2$ we have
\[
1 = r_s(0) \geq r_s(a)\geq r_s(9/16) = \frac{7}{16}-\frac{11s}{100\pi} \geq \frac{7}{16}-\frac{11}{50\pi} >0\,.
\]
Therefore, keeping in mind that $f_s=r_s^2$,
\begin{align*}
f_s'=2r_sr_s'<0\,, && f_s''=2(r_s')^2+2r_sr_s''>0\,, && f_s'''=6r_s'r_s''+2r_sr_s'''<0\,.
\end{align*}
That is, $f_s$ is positive, decreasing and convex, with $f_s'$ concave; in particular $f_s''$ is decreasing.

\step{II}{$g'(0^+)=-\infty$, $g''$ is decreasing on $[0,9/16]$, and $g''$ vanishes at a unique $a_0\in(0,0.16)$.}
The definition~(\ref{defg}) of $g$ gives
\[\begin{split}
g'(a) &=\frac{3}{2}\sqrt{1+a}\,\big(2-\sqrt{1+a}-\sqrt a\big)
-\frac{(1+a)^{3/2}}{2}\left(\frac{1}{\sqrt{1+a}}+\frac{1}{\sqrt a}\right)\\
&= \frac{\sqrt{1+a}}{2}\left(6-4\sqrt{1+a}-4\sqrt a-\frac{1}{\sqrt a}\right)\,,
\end{split}\]
so $g'(a)\to-\infty$ as $a\to0^+$. Differentiating once more gives
\[
g''(a)=\frac{3}{2\sqrt{1+a}}-\sqrt{\frac{a}{1+a}}-\frac{1}{4\sqrt{a(1+a)}}
-\sqrt{\frac{1+a}{a}}+\frac{1}{4a}\sqrt{\frac{1+a}{a}}-2 .
\]

Now, we study the sign of $g''$ and $g'''$. Let $t := t(a) = \sqrt{\tfrac{a}{1+a}}$, which is strictly increasing and maps $[0,9/16]$ onto $[0,3/5]$. Let $\tilde g$ be defined as
\[
\tilde g(t(a)) := g''(a) = \frac{3}{2}\sqrt{1-t^2}-\frac{3}{4}t-\frac{3}{2t}+\frac{1}{4t^3}-2\,.
\]
Notice that this expression diverges to $+\infty$ for $t\to 0^+$ (so for $a\to 0^+$), thus it is positive in a right neighborhood of $0$.
By computing the first derivative in $t$ we obtain
\begin{equation}
\tilde g'(t)=-\frac{3}{2}\,\frac{t}{\sqrt{1-t^2}}-\frac{3}{4t^4}+\frac{3}{2t^2}-\frac{3}{4}\leq -\frac{3}{4t^4}+\frac{3}{2t^2}
\end{equation}
which is negative for all $t \leq 3/5$. This implies that $\tilde{g}$ is decreasing in $t$ and has at most one zero;
moreover, we have
\[
\tilde{g}(0.37)\approx -0.002<0,
\]
thus there exists a unique $t_0\in (0,0.37)$ such that $\tilde{g}(t_0)=0$. Since $t(0.16)\approx 0.371>0.37$, this implies that there exists a unique $a_0\in (0,0.16)$ such that $g''(a_0)=0$. Lastly, since $t(a)$ is strictly increasing, $g'''(a)$ has the same sign as $\tilde{g}'(t)$, thus proving that $g''(a)$ is decreasing for $a\in [0,9/16]$, as we wanted to show.

\step{III}{Conclusion.}
Keep in mind that our goal is to show that $f_s\geq g$ for $a \in (0,0.14)$ if $s=2$, and for $a\in (0,9/16)$ if $s=0.62$. We set $h=f_s-g$, so that we have to establish $h\geq 0$. Notice that $h(0)=0$ since $f_s(0)=g(0)=1$. Since $g'(a)\to-\infty$ when $a\to 0^+$, while $f_s'(0)=-2-\tfrac{s}{\pi}$, we have $h'(0^+)=+\infty$, so $h$ is increasing in a right neighbourhood of $0$, hence the desired inequality $h\geq 0$ is true at least in a right neighborhood of $0$. We are now going to consider separately the case $s=2$ and the case $s=0.62$; in both cases, we will extensively use the fact that $f_s''$ and $g''$ are decreasing, as established in Steps~I and~II.

\step{IIIa}{The case $s=2$ on $(0,0.14)$.}
On $(0,0.08)$, monotonicity of $f_2''$ and $g''$ gives
\[
h''=f_2''-g''\le f_2''(0)-g''(0.08)\le-0.4<0,
\]
so $h$ is concave on $(0,0.08)$. Since $h(0)=0$ and $h(0.08)\geq 0.045$, we deduce that $h>0$ on the whole interval $(0,0.08]$.\par

On the interval $(0.08,0.14)$, the concavity of $f_2'$ and of $g'$ give
\[
f_2'(a)\leq f_2'(0.28) + f_2''(0.28)(a-0.28) \approx 3.11a-2.34
\]
and
\[
g'(a)\geq \frac{g'(0.28)-g'(0.08)}{0.2}(a-0.08)+g'(0.08)\approx 0.17a-1.45\,.
\]
Since $3.11a-2.34<0.17a-1.48$ for all $a \in (0.08, 0.14)$, we deduce that $h$ is strictly decreasing in that interval, and so for every $a$ in the interval we have
\[
h(a)\geq h(0.14)\geq 0.001>0\,,
\]
thus concluding this case.

\step{IIIb}{The case $s=0.62$ on $(0,9/16)$.}
On $[0,9/16]$ one has $\psi\ge0$ and $r_s\in(0,1]$, so $s\mapsto f_s(a)=r_s(a)^2$ is decreasing; thus $f_{0.62}\ge f_2$, and then we already have $h>0$ on $(0,0.14)$ by Step~IIIa.\par

Let us then now consider the interval $[0.14,9/16)$; first of all, we can show that $h$ is convex there. This is clear in $[a_0,9/16)$, being $a_0$ given in Step~II; in fact, in $[a_0,9/16)$ we have that $h''=f_{0.62}''-g''>0$. Instead, in $[0.14,a_0)\subseteq [0.14,0.16)$ we can use that $f_{0.62}''$ and $g''$ are decreasing, and then
\[
h''(a)=f_{0.62}''(a)-g''(a) \geq f_{0.62}''(0.16)-g''(0.14)\ \geq\ 1.9>0\,.
\]
The convexity ensures that, in the interval $[0.14,9/16)$, $h$ lies above its tangent at any point, in particular at the point $a=0.285$. This gives us that, for every $a\in[0.14,9/16]$, one has
\[
h(a)\geq  h(0.285)+h'(0.285)(a-0.285) \geq 0.0021-0.0009|a-0.285| > 0.0021 - 0.0009 >0\,,
\]
which finally concludes the proof.
\end{proof}

\section*{Acknowledgments}
C.~G. and A.~P. are members of the Gruppo Nazionale per l'Analisi Matematica, la Probabilit\`a e le loro Applicazioni (GNAMPA), which is part of the Istituto Nazionale di Alta Matematica (INdAM).
C.~G. was supported by the INdAM - GNAMPA Project, CUP E53C25002010001.
A.~P. was supported by PRIN Project 2022E9CF89, ``Geometric Evolution Problems and Shape Optimization (GEPSO)'', PNRR Italia Domani, financed by the European Union via the Program NextGeneration EU. 
C.~T.-L. has received funding from the European Research Council (ERC) under the Grant Agreement No. 862342, from AEI project PID2024-156429NB-I00 (Spain), and from the Grant CEX2023-001347-S funded by MICIU/AEI/10.13039/50110001\\1033 (Spain).

\end{document}